\documentclass[a4paper, 12pt]{article}
\usepackage[utf8]{inputenc}
\usepackage[shortlabels]{enumitem}
\usepackage{amsmath}
\usepackage{amssymb}
\usepackage{theorem}
\usepackage{tikz-cd}
\usepackage{tikz}
\usetikzlibrary{matrix}
\usepackage{afterpage}

\theoremstyle{change}  
\newtheorem{theorem}{Theorem}[section] 
\newtheorem{lem}[theorem]{Lemma}  
\newtheorem{prop}[theorem]{Proposition}
\newtheorem{coro}[theorem]{Corollary}
\theorembodyfont{\rmfamily}  
\newtheorem{remarki}[theorem]{Remark}
\newtheorem{example}[theorem]{Example}

\newtheorem{defi}[theorem]{Definition}
\newtheorem{notation}[theorem]{Notation}
\newtheorem{nothing}[theorem]{} 

\newenvironment{proof}{\noindent{\bf Proof}\ }{\qed\bigskip}

\renewcommand{\le}{\leqslant} 
\renewcommand{\ge}{\geqslant}

\newcommand{\Iso}{\mathrm{Iso}}
\newcommand{\G}{\mathcal{G}_G}

\newcommand{\R}{\mathcal{R}}
\newcommand{\qed}{\nobreak\hfill
                  \vbox{\hrule\hbox{\vrule\hbox to 5pt
                  {\vbox to 8pt{\vfil}\hfil}\vrule}\hrule}}

\newcommand{\Aut}{\mathrm{Aut}}
\newcommand{\calC}{\mathcal{C}}

\newcommand{\calF}{\mathcal{F}}

\newcommand{\calP}{\mathcal{P}}

\newcommand{\catfont}{\mathsf}

\newcommand{\Def}{\mathrm{Def}}

\newcommand{\End}{\mathrm{End}}

\newcommand{\Gset}{\lset{G}}
\newcommand{\Gmor}{\lmor{G}}
\newcommand{\GmorGal}{\lmor{G}^{\catfont{Gal}}}
\newcommand{\HmorGal}{\lmor{H}^{\catfont{Gal}}}

\newcommand{\Hom}{\mathrm{Hom}}

\newcommand{\id}{\mathrm{id}}

\newcommand{\Ind}{\mathrm{Ind}}

\newcommand{\Inf}{\mathrm{Inf}}
\newcommand{\Inn}{\mathrm{Inn}}

\newcommand{\Irr}{\mathrm{Irr}}
\newcommand{\Isom}{\mathrm{Iso}}

\newcommand{\lset}[1]{\llap{\phantom{|}}_{#1}\catfont{set}}
\newcommand{\lmor}[1]{\llap{\phantom{|}}_{#1}\catfont{mor}}

\newcommand{\Out}{\mathrm{Out}}

\newcommand{\Res}{\mathrm{Res}}

\title{Simple Section Biset Functors}
\author{\small Olcay Co\c{s}kun\thanks{Both authors are supported by T\"ubitak-1001-119F422.}\\
        \small Department of Mathematics\\
        \small Bo\u{g}azi\c{c}i University\\
        \small Bebek, 34342 Istanbul\\
        \small Turkey\\
        \small olcay.coskun@boun.edu.tr
        \and
        \small Ruslan Muslumov\\
        \small Department of Mathematics\\
        \small Bo\u{g}azi\c{c}i University\\
        \small Bebek, 34342 Istanbul\\
        \small Turkey\\
        \small ruslan.muslumov@boun.edu.tr%
        }

\begin{document}

\maketitle

\begin{abstract}
We classify simple modules over the Green biset functor of section Burnside rings. 
\end{abstract}
 
\section{Introduction}

Green biset functors, as introduced by Serge Bouc \cite{biset functor}, are rings in the category of biset functors. The initial  example is the Burnside functor. Its module category is equivalent to the category of all biset functors. 
Other examples of Green biset functors include the functor of fibered Burnside rings \cite{fibered biset}, the functor of trivial source modules \cite{Ducellier2015}, the functor of characters \cite{romero} and the functors of slice and section Burnside rings \cite{slice ring}. It is also possible to produce new Green biset functors by applying the Yoneda-Dress construction, as done in \cite{biset functor}. One of the main problems regarding Green biset functors is the classification of their simple modules. The case of biset functors is solved by Bouc in his fundamental paper \cite{densembles}. Later simple functors over the functor of rational characters is classified by Barker in \cite{barker}, simple fibered biset functors are described in \cite{fibered biset}, see also \cite{romero}, and simple functors over the trivial source ring in \cite{Ducellier2015}. There are some partial results about the functor of slice Burnside rings in \cite{Tounkara}. In contrast to all these examples, there is no general classification theorem or conjecture for simple modules over Green biset functors.

The aim of this paper is to solve the classification problem for the functor of section Burnside rings. Section Burnside rings are introduced by Bouc in \cite{slice ring} as the Grothendieck ring of the subcategory of `Galois' morphisms in the category of morphisms of $G$-sets. 
It is shown in \cite{slice ring} that basic morphisms of $G$-sets are given by the natural maps $G/S \to G/T$ associated to subgroups $S\le T\le G$ and basic Galois morphisms are those associated to
subgroups $S\unlhd T\le G$ (a pair of this from is called a section of $G$). With this result, the section Burnside ring $\Gamma(G)$ is free as an abelian group on the $G$-conjugacy classes of sections of $G$. It is also shown in \cite{slice ring} that the functor $\Gamma$ associating $\Gamma(G)$ to the finite group $G$ induces a Green biset functor. We summarize these results in Section 3. 

In this paper we use techniques from \cite{fibered biset} to analyze the structure of the Green biset functor $\Gamma$ and classify its simple modules. For this aim, we introduce a composition product for sections of direct products of groups. More precisely, given finite groups $G, H$ and $K$ and sections $S\unlhd T\le G\times H$ and $U\unlhd V\le H\times K$, we define
\[
\Big( \frac{G\times H}{S\unlhd T}\Big)\times_H \Big( \frac{H\times K}{U\unlhd V}\Big) = \bigsqcup_{t\in p_2(S) \backslash H/p_1(U)}
\Big( \frac{G\times K}{S\ast {}^{(t, 1)}U  \unlhd T\ast {}^{(t, 1)}V}\Big)
\]
where we denote the projection morphism $(G\times H)/S\to (G\times H)/T$ by $\big( \frac{G\times H}{S\unlhd T}\big)$ and the $\ast$-product is the usual composition of relations. In Section 4 we show that the category of modules over the Green biset functor $\Gamma$ is equivalent to the category of functors from the category $\mathcal P_{k\Gamma}$ of finite groups in which composition is given by the linear extension of the above composition product. Here $k$ is a commutative ring of coefficients for $\Gamma$. We call a module over $k\Gamma$ a section biset functor (over $k$).

By the general theory described in \cite{biset functor}, evaluation of a (simple) section biset functor at a finite group $G$ is (zero or) a (simple) module over the endomorphism ring $E_G:=\Gamma(G,G)$ of $G$ in $\mathcal P_\Gamma$. Moreover for a simple section biset functor $S$, if $G$ is a group of smallest order with $S(G)\not = 0$, then $S(G)$ is a simple module over the algebra $\bar E_G$. Here $\bar E_G$ is the quotient of $E_G$ by the ideal generated by all the morphisms that factor through a group of smaller order.

As also discussed in \cite{romero}, the above observation is valid for a simple module over an arbitrary Green biset functor $A$. Indeed if we define $\bar A_G$ as the quotient of $A(G, G)$ by the ideal generated by all morphisms in $\mathcal P_A$ that factor through a group of smaller order, then there is a simple $A$-module corresponding to any simple $\bar A_G$-module. To complete the classification of simple $A$-modules, one needs a classification of simple $\bar A_G$-modules for all $G$ and an equivalence relation on all these simple modules to determine isomorphism types of simple $A$-modules associated to them. 

Following \cite{fibered biset} we use certain idempotent elements in $E_G$ to prove that the algebra $\bar E_G$ is isomorphic to a product of matrix algebras over certain group algebras associated to these idempotents. In our case, the idempotents of interest in $E_G$ are parameterized by pairs $(K, P)$ of normal subgroups of $G$ that centralizes each other which are also reduced in the sense of Section 6. Given such a pair $(K,P)$, the section $(\Delta_K(G), \Delta(P))$ of $G\times G$ gives the idempotent element $e_{(K, P)}$ in $E_G$. Here $\Delta_K(G) =\{ (g, h)\in G\times G | h^{-1}g\in K \}$ and $\Delta (P)$ is the diagonal inclusion of $P$ in $G\times G$. In particular we associate simple $E_G$-modules to each reduced pair $(K, P)$. As mentioned in Section 3, it is also possible to associate crossed modules to pairs $(K, P)$ in a natural way and it turns out that two reduced pairs induce isomorphic $E_G$-modules if and only if the corresponding crossed modules are isomorphic.
With this result, we classify simple modules over the algebra $E_G$, see Section 6 for details. Finally, in Section 7, we show that the relation defined on reduced pairs can be extended to an equivalence relation on the set of quadruples $(G, K, P, [V])$ where $G$ is a finite group, $(K, P)$ is a reduced pair for $G$ and $V$ is an irreducible $E_G$-module associated to the pair $(K, P)$. Moreover there is a bijective correspondence between the equivalence classes of these quadruples and the isomorphism classes of simple section biset functors. 
 
 In order to complete the parametrization one needs to know the complete set of reduced pairs for each finite group $G$. Although we do not know the complete answer, we have some necessary and also some sufficient conditions to ensure that a pair $(K, P)$ is reduced. For example $(K, P)$ is reduced if $K\le P$. Hence the pairs $(1, P)$ for any $P\unlhd G$ and the pairs $(K, G)$ for any $K\le Z(G)$ are reduced. In particular, the algebra $\bar E_G$ is non-zero for any finite group $G$. 

Our reason for attempting to classify simple modules over the Green biset functor $k\Gamma$ is that it is different from the previous known examples in the following sense. First the previous examples are known to have strong connections with the theory of representations of finite groups and in the classification, these connections are used in crucial ways. On the other hand section Burnside rings are relatively new and these kind of deep connections are not known yet. Therefore it is close to be an abstract example. Having a successful application of the techniques from \cite{fibered biset} signals a path towards a more general theory of simple modules over Green biset functors. 

A result that might be of general interest is the version of Goursat's Theorem for sections. The well-known Goursat's Theorem for subgroups states that there is a bijective correspondence between subgroups of a direct product $G\times H$ and the quintuples $(P, K, \eta, L, Q)$ where $K\unlhd P\le G$ and $L\unlhd Q\le H$ and $\eta: Q/L\to P/K$ a group isomorphism. Our version, given as Theorem \ref{thm:Goursat}, determines exact conditions on the pairs of quintuples coming from Goursat's Theorem so that the induced map between sections of $G\times H$ and the pairs satisfying these conditions is bijective.

\section{Green biset functors}
Let $k$ be a commutative ring with unity and let $\calC_k$ denote the biset category over $k$, that is, the category with objects consisting of finite groups and morphisms from $G$ to $H$ are given by the double Burnside group $kB(H, G)$. A $k$-linear functor $F: \calC_k\to k$-mod is called a \emph{biset functor over $k$}. We denote the functor category of all biset functors over $k$ by $\calF_k$. 

\begin{defi}Let $A$ be a biset functor over $k$. We call $A$ a \emph{Green biset functor} over $k$ if there are bilinear maps $A(G)\times A(H)\to A(G\times H), (a, b)\mapsto a\times b$ for each pair of finite groups $G$ and $H$ and an element $\epsilon_A\in A(1)$ with the following properties.
\begin{enumerate}
	\item[(i)] For any finite groups $G, H$ and $K$ and any $(a, b, c)\in
	A(G)\times A(H)\times A(K)$, we have $$(a\times b)\times c = A(\Isom^{\alpha})(a\times(b\times c))$$ where $\alpha$ is the canonical isomorphism from $G\times(H\times K)$ to $(G\times H)\times K$. 
	\item[(ii)] For any finite group $G$ and any $a\in A(G)$, we have $$A(\Isom^\lambda)(\epsilon_A\times a) = a = A(\Isom^{\lambda'})(a\times \epsilon_A)$$ where $\lambda: 1\times G\to G$ and $\lambda': G\times 1\to G$ denote the canonical isomorphisms.
	\item[(iii)] Let $X\in kB(G', G)$ and $Y\in kB(H', H)$. For any $(a,b)\in A(G)\times A(H)$, we have $$A(X\times Y)(a\times b) = A(X)(a)\times A(Y)(b).$$ 
\end{enumerate}
\end{defi}
In other words, a Green biset functor is a monoid object in the category $\calF_k$ of all biset functors. The bilinear maps defining the Green biset functor structure on $A$ gives a ring structure on $A(G)$ for any finite group $G$ and these structures are compatible with the biset functor structure on $A$ via the above properties. Conversely, if $A$ is any biset functor such that $A(G)$ has a ring structure for each $G$ compatible with the biset functor structure, one can make $A$ a Green biset functor. See \cite[3.2.7 and 4.2.3]{romero}.

\begin{example} The Burnside biset functor is a Green biset functor and it is an initial object in the category of all Green biset functors \cite[8.5.3]{biset functor} . Given a Green biset functor $A$, the unique morphism $kB\to A$ of Green biset functors is given by mapping a $G$-set $X$ to $A(X)(\epsilon_A)$ in $A(G)$. Here we regard $X$ as a $(G, 1)$-biset. Other known examples of Green biset functors include the functor of representation rings, the functor of fibered Burnside rings and the functor of trivial source rings.  
\end{example}

\begin{nothing} {\bf The category $\calP_A$.}\quad Let $A$ be a Green biset functor. We associate a category $\calP_A$ to $A$ as follows (\cite[Section 8.6]{biset functor}).  The objects of $\calP_A$ are all finite groups. 
\begin{enumerate}
	\item[(i)] Given finite groups $G$ and $H$, we put $\Hom_{\calP_A}(H, G)= A(G\times H)$.
	\item[(ii)] Given finite groups $G, H, K$, and morphisms $\alpha\in A(H\times K)$ and  $\beta\in A(G\times H)$, the composition is defined by
	$$\beta\circ \alpha = A(\Def^{G\times \Delta(H)\times K}_{G\times K}\Res^{G\times H\times H\times K}_{G\times \Delta(H)\times K})(\beta\times \alpha).$$ Here we identify $G\times K$ with $(G\times \Delta(H)\times K)/(1\times \Delta(H)\times 1)$ via the obvious canonical isomorphism. 
	\item[(iii)] Given a finite group $G$, the identity morphism of $G$ in $\calP_A$ is $A(\Ind_{\Delta(G)}^{G\times G}\Inf_1^{\Delta(G)})(\epsilon_A)$.
\end{enumerate}

Note that the unique morphism $kB\to A$ of Green biset functors induces a functor $\mathcal C_k\to \mathcal P_A$. In particular we may consider the basic operations associated to bisets as operations in $\mathcal P_A$. 
\end{nothing}

\begin{nothing}{\bf $A$-modules.}\label{sec:a-mods}\quad With this definition, an \emph{$A$-module} is a $k$-linear functor $\calP_A\to k$-mod. We denote the (abelian) functor category of all $A$-modules by $A$-Mod. When $A=kB$, the category $kB$-Mod is the same as the category $\calF_k$ of all biset functors. Moreover for any Green biset functor $A$, via the unique morphism $kB\to A$ of Green biset functors, any $A$-module is a biset functor. 

Let $A$ be a Green biset functor and $F$ be an $A$-module. For any finite group $G$, the evaluation $F(G)$ is a module over the endomorphism ring 
$E_G^A = \End_{\calP_A}(G)$ of $G$ in $\calP_A$. Moreover if $G$ is a group of minimal order such that $F(G)\not =0$, then the $E_G^A$-module $F(G)$ is annihilated by all the endomorphisms that factor through a group of smaller order. We denote the ideal of $E_G^A$ generated by such morphism by $I_G^A$ and the quotient $E_G^A/I_G^A$ by $\hat A(G)$ and call it the \emph{essential algebra of $A$ at $G$}. In particular when $G$ is minimal for $F$, then $F(G)$ becomes an $\hat A(G)$-module.

Let $S$ be a simple $A$-module. Then for any group $G$, the $E_G^A$-module $S(G)$ is either zero or a simple $E_G^A$-module. If $G$ is minimal for $S$ then $S(G)$ becomes a simple $\hat A(G)$-module. By the general results in \cite{biset functor}, the simple $A$-module $S$ is determined by the pair $(G, S(G))$ for any finite group $G$. When $G$ is also minimal for $S$, we call the pair $(G, S(G))$ a \emph{seed} for $S$. If $(G, V)$ is a seed for $S$, we also denote $S$ by $S_{G, V}^A$. Conversely, if a pair $(G, V)$ is given where $V$ is a simple $\hat A(G)$-module, there is a unique simple $A$-module with seed $(G, V)$. We denote it by $S_{G, V}^A$. In general two seeds $(G, V)$ and $(G', V')$ may induce isomorphic simple $A$-modules. In this case, we say that $(G, V)$ and $(G', V')$ are equivalent. Then we obtain a correspondence between isomorphism classes of simple $A$-modules and equivalence classes of seeds. Note that two seeds $(G, V)$ and $(G', V')$ might be equivalent without $G$ and $G'$ being isomorphic, see \cite{fibered biset}.

Finally we recall from \cite{fibered biset} and \cite{biset functor} the following construction of the simple functors from seeds. Let $V$ be an $E_G$-module. We denote by $L_{G, V}$ the $A$-module given by
\[
L_{G, V}(H) = A(H\times G) \otimes_{E_G}V
\]
for any finite group $H$. Here the right action of $E_G$ is given by composition on the right and the action of $A$ is given by composition on the left. 

Now suppose $V$ is a simple $E_G$-module. Then by \cite[Lemme 1]{densembles}, the $A$-module $L_{G, V}$ has a unique maximal submodule $J_{G, V}$ given by
\[
J_{G, V} (H) = \{ \sum_i x_i\otimes v_i \in L_{G, V}(H) | \forall y\in A(G\times H): \sum_i (y\circ x_i)v_i = 0 \}.
\]
The simple head $S_{G, V}$, which is the quotient $L_{G, V}/J_{G, V}$ and a simple object of $\calF_k$, has the property that $S_{G, V}(G) \cong V$.
\end{nothing}
\begin{nothing}{\bf Covering algebras.} In general determination of the seeds and the equivalence relation is not easy. The first step in this direction is to identify essential algebras for $A$. Suppose $E'_G$ is a subalgebra of $E_G^A$ such that 
\[
E_G^A = E'_G + I_G^A.
\]
Then we get 
\[
\hat{A}(G) = E'_G /(E'_G\cap I_G^A).
\]
In other words, one can consider $E'_G$ as an approximation to the essential algebra $\hat{A}(G)$. We call 
$E'_G$ a {\it covering algebra} for $E_G^A$. 

Suppose $|G|$ is invertible in $k$. Then an example of covering algebra is the truncated algebra $\tilde{e}_G^G A(G\times G)\tilde{e}_G^G$. Here $\tilde{e}_G^G$ is the image of the primitive idempotent $e_G^G$ of the Burnside ring $B(G)$ in the double Burnside ring $B(G\times G)$ as defined in \cite[Section 6.5]{biset functor}. We regard it as a morphism in $\mathcal P_A$ as explained above. This is a covering algebra, by \cite[Proposition 6.5.5]{biset functor}. 
The algebra $E_G^c$ introduced in \cite[Section 6]{fibered biset} for fibered Burnside functors is equal to 
$\tilde{e}_G^G E_G\tilde{e}_G^G$ by \cite[Proposition 6.5.5]{biset functor}.
It is possible to introduce other covering algebras. One aims to choose a covering algebra for which the intersection with the ideal $I_G^A$ is easier to determine.
\end{nothing}

\begin{nothing}{\bf Idempotents in covering algebras.} \label{app:ipots} In \cite{fibered biset} the covering algebra for the fibered biset functors is studied through a set of central idempotents. This approach can be abstracted to Green biset functors as follows. Note that it only holds under an assumption and generalizes some very basic results.

We fix a covering algebra $E_G^c$ for $A$. Let $(X, \le)$ be a finite poset with minimum element $x_0$. Suppose there is a function
\[
e : X \to E_G^c, x\mapsto e_x
\]
such that 
\begin{enumerate}
\item[(a)] $e_{x_0} = 1_{E_G^c}$
\item[(b)] $ e_x \cdot e_y = \begin{cases} 
       e_{x\vee y} & \mbox{\rm if} \quad x\vee y \, \mbox{\rm exists} \\
      0 & \mbox{\rm otherwise}
   \end{cases}$
\end{enumerate}
In particular $e_x$ is an idempotent in $E_G^c$ for each $x\in X$. We define
\[
f_x = \sum_{x\le y\in X}\mu_{x, y} e_y
\]
where $\mu_{x, y}$ is the Mobius function of the poset $X$. Then by the Mobius inversion formula, we also get
\[
e_x = \sum_{x\le y\in X} f_y
\quad \mbox{\rm and} \quad
1_{E_G^c} = e_{x_0} = \sum_{x\in X} f_x.
\]

The following result is a generalization of Proposition 4.4 from \cite{fibered biset} and Theorem 10.1 from \cite{essential relations}.  
\begin{prop}\label{prop:ipot-ortho}
For all $x, y\in X$, we have
\[ e_x\cdot f_y = f_y\cdot e_x = \begin{cases} 
       f_y & \mbox{\rm if} \quad x\leq y \\
      0 & \mbox{\rm otherwise}
   \end{cases}
\] and
\[ f_x\cdot f_y = \begin{cases} 
       f_x & \mbox{\rm if} \quad x= y \\
      0 & \mbox{\rm otherwise}
   \end{cases}
\] 
\end{prop}
\begin{proof}
The proof is almost identical to the proof of \cite[Proposition 4.4]{fibered biset}. We repeat it for completeness. For the first case, if $x\le y$,
then 
\begin{eqnarray*}
e_x\cdot f_y &=& \sum_{y\le z} \mu_{y, z}e_x\cdot e_z\\
&=& \sum_{y\le z} \mu_{y, z} e_z = f_y.
\end{eqnarray*}
Here the second equality holds since $x\le z$ implies $e_x\cdot e_z = e_z$. Also it is clear that $e_x$ and 
$f_y$ commutes. For the rest of the claims, we argue by induction on $d = d_x + d_y$ where for any $x\in X$, the natural number $d_x$ is the largest $n\in \mathbb N_0$ such that there exists a chain 
$x = a_0 < a_1 < \cdots < a_n$ in $X$. Now if $d = 0$, then both $x$ and $y$ are maximal in $X$, hence we 
have $f_y = e_x$. Therefore either $x = y$ and $x\vee y  = x$ and hence $e_x \cdot f_y =e_x = f_y$ or their join does not exist and hence $e_x\cdot f_y = 0$.

Next suppose $d>0$ and that the proposition holds for all smaller $d$. We claim that if $f_x\cdot f_y$ is non-zero, then $x=y$. Indeed we have
\begin{eqnarray*}
f_x\cdot f_y &=& \sum_{x\le z, y\le t}\mu_{x, z}\mu_{y, t} e_z\cdot e_t. 
\end{eqnarray*}
Here the product $e_z\cdot e_t$ is non-zero if and only if the join $z\vee t$ exists, in which case, $x\vee y$ also
exists and $x\vee y\le z\vee t$. Hence the above equality becomes
\begin{eqnarray*}
f_x\cdot f_y &=& \sum_{x\vee y\le z}\mu_{x, z}\mu_{y, z} e_z. 
\end{eqnarray*}
In particular $e_{x\vee y}\cdot f_x\cdot f_y = f_x\cdot f_y \not = 0$. Hence by the first part, either $x < x\vee y$ or $y< x\vee y$ and we get $e_{x\vee y}\cdot f_x\cdot f_y = 0$ by induction. This is a contradiction, so we must have $x=y$. 

Finally if $x = y$, then
\[
e_x = e_x^2 = f_x^2 + \sum_{x<y} f_y
\]
and hence $f_x^2 = f_x$ as required.
\end{proof}

Note that the elements $e_x, x\in X$ generates a commutative subalgebra $\bigoplus_{x\in X} \mathbb Z e_x$ 
of $E_G^c$ and the set $\{f_x | x\in X\}$ is a basis of this algebra consisting of orthogonal idempotents summing up to the identity element.  
\end{nothing}

\begin{nothing}
To resolve the structure of $E_G^c$ further, one needs more properties of the idempotents introduced above. The properties in \cite{fibered biset} depend on fibered bisets and it is not clear to us how to make an abstraction to the above case. In this paper we consider the special case of section Burnside functors \cite{slice ring}. 
\end{nothing}

\section{Section Burnside rings}
In this section we collect together definitions and basic results on section Burnside rings as defined in \cite{slice ring}. We also prove a version of Goursat Theorem for sections of direct products of groups. 

\begin{nothing}{\bf The categories $\Gmor$ and $\GmorGal$.}\quad
Let $G$ be a finite group. We denote the category of finite $G$-sets and $G$-equivariant morphisms by $\Gset$. For any pair of morphisms $f: X\to Y$ and $f^\prime: X^\prime \to Y^\prime$, a {\it morphism} from $f$ to $f^\prime$ is defined to be a pair $(\alpha: X\to X^\prime, \beta: Y\to Y^\prime)$ of morphisms of $G$-sets such that the following diagram commutes.
   \begin{center}
    
    \begin{tikzpicture}
        \matrix (m) [matrix of math nodes,row sep=3em,column sep=4em,minimum width=2em]
        {
            X & Y \\
             X^\prime & Y^\prime \\
        };
  \path[-stealth]
    (m-1-1) edge node [left] {$\alpha$} (m-2-1)
            edge node [above] {$f$} (m-1-2)
    (m-2-1.east|-m-2-2) edge node [below] {}
            node [above] {$f^\prime$} (m-2-2)
    (m-1-2) edge node [right] {$\beta$} (m-2-2);
    \end{tikzpicture}
    \end{center}
For $(\alpha, \beta): f\to f^\prime$ and $(\gamma, \delta): f^\prime \to f^{\prime\prime}$, we define the {\it composition} $(\gamma, \delta)\circ (\alpha, \beta)$ component-wise. Notice that the pair $(\id_X , \id_Y)$ becomes the identity morphisms of $f:X\to Y$. With this data, we form the category $\Gmor$ of morphisms of $G$-sets. By \cite[Proposition 2.2]{slice ring}, the disjoint union and the direct product of $G$-sets induce a coproduct and a product in $\Gmor$, respectively. 

A morphism $f: X\to Y$ is called a {\it Galois morphism} if for any $x, x^\prime \in X$ with $f(x) = f(x')$, we have
$f = f\circ \phi$ with $\phi(x) = x'$ for some automorphism $\phi$ of $X$ in $\Gset$. By \cite[Section 9]{slice ring}, the class of all Galois morphisms of morphisms of $G$-sets forms a subcategory of $\Gmor$ and the product and coproduct in $\Gmor$ restricts to a product and a coproduct in this subcategory. We denote the category of Galois morphisms by $\GmorGal$.

Isomorphisms of morphisms of $G$-sets are defined in the usual way and we denote the isomorphism class of the morphism $f: X\to Y$ by $[f: X\to Y]$. Given a slice $S\le T$ of $G$, we denote the isomorphism class of the projection morphism $G/S\to G/T$ by $\langle G/T, G/S \rangle$. Note that by \cite[Remark 9.5]{slice ring}, the morphism $G/S\to G/T$ is a Galois morphism if and only if $S\unlhd T$. 
\end{nothing}

\begin{nothing}{\bf Slice and section Burnside rings.}\quad
Following Bouc \cite{slice ring}, we define the {\it slice Burnside ring} $\Xi(G)$ of $G$ as the quotient of the free abelian group on the isomorphism classes $[f: X\to Y]$ of morphisms of $G$-sets by the subgroup generated by the relation
   \begin{equation*}
        [f_1\sqcup f_2:(X_1\sqcup X_2){\rightarrow} Y]-[f_1:X_1{\rightarrow} f(X_1)]-[f_2: X_2{\rightarrow} f(X_2)]
    \end{equation*}
whenever $f: X \rightarrow Y$ is a morphism of finite $G$-sets with a decomposition $X = X_1\sqcup X_2$ as a disjoint union of $G$-sets, where $f_1 = f_{|X_1}$ and $f_2 = f_{|X_2}$. The product is induced by the product of morphisms.
The image of $f$ in $\Xi(G)$ is denoted by $[f]_G$ and to simplify the notation further, we denote $\langle G/T, G/S \rangle$ by $\langle T, S \rangle_G$.

The subring of $\Xi(G)$ generated by the isomorphism classes of Galois morphisms of $G$-sets is called the 
{\it section Burnside ring} and is denoted by $\Gamma(G)$.

With this notation, by \cite[Lemma 3.3]{slice ring}, the zero element in $\Xi(G)$ is $[\emptyset\to \emptyset]_G$ and we have $[f: X\to Y]_G = [f: X\to f(X)]_G$ for any morphism $f:X \to Y$ of $G$-sets. Furthermore we have
\[
[f_1\sqcup f_2: X_1\sqcup X_2\to Y_1\sqcup Y_2]_G = [f_1: X_1\to Y_1]_G + [f_2: X_2\to  Y_2]_G
\]
In particular, by \cite[Lemma 3.4]{slice ring}, for any morphism $f:X\to Y$, there is a decomposition 
\[
[f]_G = \sum_{x\in [G\backslash X]} \langle G_{f(x)}, G_x \rangle_G
\]
where the sum is over a complete set of representatives of $G$-orbits in $X$. Therefore the slice Burnside ring $\Xi(G)$ has a basis 
$\langle T, S\rangle_G$ as $(T, S)$ runs over a complete set of representatives of $G$-conjugacy classes of pairs $(T, S)$ of subgroups $S\le T\le G$, see \cite[Theorem 4.6]{slice ring}. These pairs are called slices of $G$. Similarly, the subset consisting of the pairs parameterized by the representatives of the $G$-conjugacy classes of sections $(T, S)$ of $G$ forms a basis for the section Burnside ring $\Gamma(G)$.
\end{nothing}

\begin{nothing}{\bf The group $\Gamma(G, H)$.} Let $H$ be a finite group. We write $\Gamma(G, H)$ for the section Burnside ring of $G\times H$. It has a basis parameterized by the representatives of $G\times H$-conjugacy classes of sections of the direct product $G\times H$. As usual we shall consider the elements of $\Gamma(G, H)$ as morphisms from $H$ to $G$, and when considering as such, 
we denote the Galois morphism  $(X\xrightarrow{g} Y)$ of $G\times H$ and its image in $\Gamma(G, H)$, respectively, by
\begin{equation*}
    \Big (\frac{G\times H}{X\xrightarrow{g} Y}\Big ) \quad \mbox{\rm and} \quad \Big [\frac{G\times H}{X\xrightarrow{g} Y}\Big ] 
\end{equation*}
In case we have a section $(T, S)$ of $G\times H$, we denote the projection $\Big (\frac{G\times H}{S}\Big )\to \Big (\frac{G\times H}{T}\Big )$ and its image in $\Gamma(G, H)$, respectively, by
\begin{equation*}
    \Big (\frac{G\times H}{S\unlhd T}\Big ) \quad \mbox{\rm and} \quad \Big [\frac{G\times H}{S\unlhd T}\Big ] 
\end{equation*}
\end{nothing}
\begin{nothing}{\bf Digression on crossed modules.}\label{app:crossed}
It turns out that crossed modules are useful to express conditions on subgroups of direct products to be sections. We recall basic definitions. See \cite{Whitehead} for further details.

Let $G$ and $A$ be groups. Assume $G$ acts on $A$, written $(g,a)\mapsto {^g}a$ for $g\in G, a\in A$ and also suppose that there is a group homomorphism $\partial : A\rightarrow G$.  The triple $(A,G,\partial)$ is called a $\textbf{crossed module}$ if the following two conditions hold.
        \begin{enumerate}[(i)]
             \item $\partial ({^g}a) = g\partial (a) g{^{-1}},$
             \item ${^{\partial(b)}}a = bab^{-1}$
        \end{enumerate}
    for all $g\in G$ and $a,b\in A$.

Let $(A, G, \partial)$ and $(A',G',\partial')$ be crossed modules. We define a {\bf morphism} from $(A, G, \partial)$ to   $(A',G',\partial')$ as a pair $(\alpha, \beta)$ of group homomorphisms $\alpha: A\to A'$ and $\beta: G\to G'$ 
satisfying
   \begin{enumerate}
\item[(a)] The diagram
    \begin{equation*}
    	\begin{tikzcd}[column sep=3em,row sep=3em]
	    A \arrow[r, "\alpha"] \arrow[d, "\partial"]
    	& A' \arrow[d, "\partial'" ] \\
    	G \arrow[r, "\beta " ]
    	& G'
        \end{tikzcd}
    \end{equation*}
    commutes.
	\item[(b)] For all $g\in G$ and $a\in A$, we have $\alpha({^g}a) = {^{\beta(g)}}\alpha(a)$.
	\end{enumerate}

The pair $(\alpha, \beta)$ is called an isomorphism, a monomorphism, an epimorphism or an automorphism of crossed modules if the homomorphisms $(\alpha, \beta)$ are both isomorphisms, monomorphisms,  epimorphisms or automorphisms, respectively. Hence the group of automorphisms of $(A,G,\partial)$ is defined and denoted by $\Aut (A,G,\partial)$.

Let $\alpha: G\to \Aut(A), g\mapsto \alpha_g$ be the action of $G$ on $A$ and $c: G\to \Aut(G), g\mapsto c_g$ be the conjugation action of $G$ on itself. For any $g\in G$ the pair $(\alpha_g, c_g)$ is an automorphism of $(A, G, \partial)$ and it is called an inner automorphism. The induced function $\theta: G\to 
\Aut(A, G, \partial)$ is a homomorphism, its image is a normal subgroup, denoted by $\Inn (A,G,\partial)$. As usual the quotient $\Out (A,G,\partial) = \Aut (A,G,\partial)/ \Inn (A,G,\partial)$ is called the group of outer automorphisms of $(A,G,\partial)$. 
\end{nothing}

\begin{nothing}{\bf Goursat Theorem for sections.}\quad
Since a basis for $\Gamma(G, H)$ is given in terms of sections of $G\times H$, we first parametrize these sections in a way as Goursat Theorem parametrizes subgroups of direct products. 
Let $U$ be a subgroup of $G\times H$. We set 
\begin{enumerate}
\item[(i)] $p_1(U) = \{g\in G \mid (g,h)\in U\, \mbox{\rm for some} \, h\in H\}$, $p_2(U) = \{h\in H \mid (g,h)\in U\, \mbox{\rm for some} \, g\in G\}$,
\item[(ii)] $k_1(U) = \{g\in G \mid (g,1)\in U\}$, $k_2(U) = \{h\in H \mid (1,h)\in U\}$,
\item[(iii)] $\eta_U: p_2(U)/k_2(U) \to p_1(U)/k_1(U), h\cdot k_2(U)\mapsto g\cdot k_1(U)$ provided $(g, h)\in U$.
\end{enumerate}
By Goursat' Theorem, the map $\eta_U$ above is a group isomorphism and the correspondence mapping $U$ to the quintuple $(p_1(U), k_1(U), \eta_U, k_2(U), p_2(U))$ establishes a bijective correspondence between the subgroups $U$ of $G\times H$ and the quintuples $(P, K, \eta, L, Q)$ with $K\unlhd P\le G$ and $L\unlhd Q\le H$ and
$\eta: Q/L \to P/K$ a group isomorphism. We call $(p_1(U), k_1(U), \eta_U, k_2(U), p_2(U))$ the {\it Goursat correspondent} of $U$. Also we call a quintuple $(P, K, \eta, L, Q)$ satisfying the above conditions a {\it Goursat quintuple}.

To determine a similar correspondence for the sections of $G\times H$, let $(T, S)$ be a section of $G\times H$ with the corresponding quintuples $(P_T, K_T, \eta_T, L_T, Q_T)$ and $(P_S, K_S, \eta_S, L_S, Q_S)$. Then by Goursat Theorem, we have
\begin{enumerate}[(i)]
    \item[(S1)] $1\le K_X\unlhd P_X\le G$ and $1\le L_X\unlhd Q_X\le H$ for $X\in \{ T, S\}$.
     \item[(S2)] $\eta_S$ and $\eta_T$ induced respectively by $S$ and $T$  are isomorphisms.
   \end{enumerate}
In addition the condition $S\unlhd T$ implies that the following statements also hold:
\begin{enumerate}
    \item[(S3)] $K_S\unlhd K_T$ and $L_S\unlhd L_T$. 
 
    \item[(S4)] $K_S, P_S\unlhd P_T$ and $L_S, Q_S\unlhd Q_T$,
   \end{enumerate}  
  Here (S3) is straightforward. To prove that $K_S\unlhd P_T$, let $x\in K_S$ and $g\in P_T$ so that $(x,1)\in S $ and $(g,h) \in T$ for some $h\in H$. Then since S is normal in T, we get $({}^{g}x,1) = {}^{(g,h)}(x,1)\in S$ In particular ${^g}x \in K_S$. Similarly if $x\in P_S$ with $(x,y)\in S$ then$({^g}x,{^h}y)=^{(g,h)}(x,y)\in S$, that is, ${^g}x\in P_S$. The rest is proved in the same way.  
 \begin{enumerate}
    \item[(S5)] $[K_T, P_S] \le K_S$ and $[L_T, Q_S] \le L_S$
  \end{enumerate}  
Indeed let $x\in P_S$ with $(x, y)\in S$ and $g\in K_T$. Note that $(g, 1)\in T$. Then since S is normal in T, we get ${}^{(g, 1)}(x,y) = ({}^gx, y)\in S$. In particular $x^{-1}\cdot{}^gx\in K_S$ which implies $xK_S={}^gx K_S$, as required. Note that this condition is equivalent to the following condition. 
 \begin{enumerate}  
    \item[(S5$'$)]   $K_T/K_S\le C_{P_T/K_S}(P_S/K_S)$ and $L_T/L_S\le C_{Q_T/L_S}(Q_S/L_S)$.
  \end{enumerate}    
  
By Condition (S4), we have that $P_T/K_S$ contains $P_S/K_S$ as a subgroup and $P_T/K_T$ as a quotient. We have the following diagram.
\[
\begin{tikzcd}
 & P_T/K_S \arrow{dr}{\pi} \\
P_S/K_S \arrow{ur}{\iota} && P_T/K_T
\end{tikzcd}
\] 
We denote the composition by $\partial = \pi\circ\iota$, it is given by $\partial(xK_S) = xK_T$. 
Together with (S5), the conjugation action of $P_T/K_S$ on $P_S/K_S$ descents to a conjugation action of $P_T/K_T$ on $P_S/K_S$. Now we get
\begin{enumerate}
    \item[(S6)] $(P_S/K_S, P_T/K_T, \partial)$ and $(Q_S/L_S, Q_T/L_T, \partial)$ are crossed modules.
\end{enumerate}  
To prove (S6) one needs to justify the two conditions given in Section \ref{app:crossed}. For the first condition,
let $gK_T\in P_T/K_T$ and $xK_S\in P_S/K_S$. Then
\[
\partial({}^{gK_T} xK_S) = \partial( {}^g xK_S) = {}^gxK_T = {}^{gK_T}\partial(xK_S),
\]
hence the first condition of the definition holds. For the second condition, let also $yK_S\in P_S/K_S$. Then
\[
{}^{\partial(xK_S)}yK_S = {}^{xK_T}yK_S = {}^xyK_S = {}^{xK_S}yK_S
\]
which proves  the second condition. Thus $(P_S/K_S, P_T/K_T, \partial)$ is a crossed module. Replacing $K$ with $L$ and $P$ with $Q$ in the above arguments, one can also prove that $(Q_S/L_S, Q_T/L_T, \partial)$ is a crossed module. With this identifications, the following also holds.
\begin{enumerate}
    \item[(S7)] $(\eta_S, \eta_T): (Q_S/L_S, Q_T/L_T, \partial)\to (P_S/K_S, P_T/K_T, \partial)$ is an isomorphism of crossed modules. 
\end{enumerate}  
Since $\eta_S$ and $\eta_T$ are already isomorphisms, we only need to check that the pair $(\eta_S, \eta_T)$ is a morphism of crossed modules, see Section \ref{app:crossed}.  Let $qL_S\in Q_S/L_S$. and let $\eta_S(qL_S) = pK_S$. Then
\[
\eta_T(\partial(qL_S)) = \eta_T(qL_T) = pK_T
\]
where the last equation holds since $(p, q)\in S$ implies it is in $T$. On the other hand
\[
\partial(\eta_S(qL_S)) = \partial(pK_S) = pK_T
\]
and hence the diagram in Section \ref{app:crossed} commutes. For the other condition, let also $(g, h)\in T$ so that $\eta_T(hL_T) = gK_T$. Then
\[
\eta_S({}^{hL_T}qL_S) = \eta_S({}^hqL_S) = {}^gpK_S.
\]
Here the last equality holds since $S\unlhd T$. We clearly have
\[
{}^gpK_S = {}^{\eta_T(hL_T)}\eta_S(qL_S).
\] 
Hence $(\eta_S, \eta_T)$ is a morphism of crossed modules and hence we proved (S7). Collecting the properties we obtain the following version of Goursat' Theorem for sections.
\begin{theorem}\label{thm:Goursat}
Let $G$ and $H$ be finite groups. Then there is a bijective correspondence between
\begin{enumerate}[(a)]
\item the set of all sections $(T,S)$ of $G\times H$ and
\item the set of all pairs $((P_1, K_1, \eta_1, L_1, Q_1), (P_2, K_2, \eta_2, L_2, Q_2))$ of Goursat quintuples satisfying the following conditions.
\begin{enumerate}
\item[(S3)] $K_2\unlhd K_1$ and $L_2\unlhd L_1$, 

    \item[(S4)] $K_2, P_2\unlhd P_1$ and $L_2, Q_2\unlhd Q_1$,
 
    \item[(S5)] $[K_1, P_2]\le K_2$ and $[L_1, Q_2] \le L_2$.
     \item[(S6)] $(P_2/K_2, P_1/K_1, \partial)$ and $(Q_2/L_2, Q_1/L_1, \partial)$ are crossed modules where $\partial$ is given by $\partial(xA_2) = xA_1$ for $A\in\{K, L\}$ and in both cases, the action is given by conjugation.
    \item[(S7)] $(\eta_2, \eta_1): (Q_2/L_2, Q_1/L_1, \partial)\to (P_2/K_2, P_1/K_1, \partial)$ is an isomorphism of crossed modules. 
\end{enumerate}
\end{enumerate} 
The correspondence is given by mapping a section $(T, S)$ of $G\times H$ to the pair of Goursat correspondents of $T$ and $S$.
\end{theorem}
\begin{proof}
By the arguments preceding the theorem, the mapping is well-defined and injective. Hence it is sufficient to prove that any element of the later set corresponds to a section in $G\times H$. Let $((P_1, K_1, \eta_1, L_1, Q_1), (P_2, K_2, \eta_2, L_2, Q_2))$ be a pair of Goursat quintuples satisfying the above conditions. Also let $(T, S)$ be the pair of subgroups of $G\times H$ corresponding to these Goursat quintuples. Explicitly we have
\[
S= \{(x, y)\in P_2\times Q_2 \mid \eta_2(yL_2) = xK_2\}\,\, \mbox{\rm and}\,\, T= \{(g, h)\in P_1\times Q_1 \mid \eta_1(hL_1) = gK_1\}
\] 
We have to show that $S\unlhd T$. It is straightforward to show that $S\subseteq T$ is equivalent to (S6) and $T\subset N_G(S)$ is equivalent to (S7). Note that although the conditions (S3)-(S4)-(S5) are not used explicitly, they are needed for the last two conditions to make sense. 
\end{proof}

\begin{nothing}{\bf Invariants of sections.}
Let $G$ and $H$ be finite groups. We denote the set of all sections of $G$ by $\Sigma_G$. Let $(T,S)\in \Sigma_{G\times H}$. We associate to $(T,S)$ its left and right invariants
$$l(T,S) := (p_1(T),k_1(T),p_1(S),k_1(S))$$ and
$$r(T,S) := (p_2(T),k_2(T),p_2(S),k_2(S))$$
respectively. In some cases, the middle invariants $\;l_0(T,S):= (k_1(T),p_1(S))$ and $r_0(T,S):=(k_2(T),p_2(S))$ will be more useful. The following result states further relations between Goursat correspondents.
\end{nothing}
\begin{prop}\label{sections prop}
Let $G$ and $H$ be finite groups, $(T,S)\in \Sigma_{G\times H}$ with Goursat correspondents 
$(P_T, K_T, \eta_T, L_T, Q_T) $ and $(P_S, K_S, \eta_S, L_S, Q_S)$.
Then there are group isomorphisms
\begin{enumerate}[(a)]
\item      \begin{equation*}
            \frac{P_S}{P_S\cap K_T} \cong \frac{Q_S}{Q_S\cap L_T},
        \end{equation*}
        
 \item       \begin{equation*}
            \frac{P_T}{P_S K_T} \cong \frac{Q_T}{Q_S L_T} 
        \end{equation*}
        and
  \item      \begin{equation*}
            \frac{P_S\cap K_T}{K_S} \cong \frac{Q_S\cap L_T}{L_S}.
        \end{equation*}
         \end{enumerate}
In particular if $P_T = G$, $K_S = 1$ and $K_T\le P_S$ then $|G|\le |H|$. 
\end{prop}

\begin{proof}  (a) First note that since $K_T$ and $P_S$ are both normal in $P_T$, the intersection $P_S\cap K_T$ is normal in $P_S$. Hence the isomorphism $\eta_S: Q_S/L_S \to P_S/K_S$ pre-composed by $\pi: Q_S\to Q_S/L_S$ and post-composed with $\pi: P_S/K_S \to P_S/(P_S\cap K_T)$ becomes a surjective homomorphism $$\tilde\eta_S: Q_S\to P_S/(P_S\cap K_T).$$ The kernel of $\tilde\eta_S$ is clearly $Q_S\cap L_T$. Note that by the second isomorphism theorem we also have the isomorphisms  
 $$P_SK_T/K_T\cong P_S/(P_S\cap K_T)\cong Q_S/(Q_S\cap L_T)\cong Q_SL_T/L_T.$$ 
        (b) As in the above case, post-compose $\eta_T$ with the canonical homomorphism $\pi: P_T/K_T\to P_T/(P_SK_T)$ to obtain 
        \[
        \tilde\eta_T : Q_T/L_T\to P_T/(P_SK_T)
        \]
Let $(p, q)\in T$. Then we have $qL_T\in \ker \tilde\eta_T$ if and only if $pK_T\in P_SK_T/K_T$. We claim that the last condition is equivalent to $qL_T\in Q_SL_T/L_T$. There is nothing to prove if $p\in K_T$. Suppose $p\in P_S$. It is sufficient to prove that $qL_T\in Q_SL_T/L_T$ and this follows directly from the compatibility of $\eta_S$ and $\eta_T$. Hence the kernel of $\tilde\eta_T$ is $Q_SL_T/Q_S$ and the result follows from the first and the third isomorphism theorems.    
        
        (c) Consider the restriction $\bar\eta_S$ of $\eta_S$ to the subgroup $(Q_S\cap L_T)/L_S$. Since 
    $\eta_S$ is an isomorphism, the restriction of $\bar\eta_S$ to its image is an isomorphism. To determine its 
    image, let $q\in Q_S\cap L_T$ and let $p\in G$ be such that $(p, q)\in S$. Then since $(1, q)\in T$, we also 
    get $(p, 1)\in T$, that is, $p\in K_T$. Therefore $pK_S$ is contained in $(P_S\cap K_T)/K_S$, as required.
     
     For the final statement, we write the order of $G$ as
     $$|G| = |G:P_T|\cdot |P_T:P_SK_T|\cdot |P_SK_T:K_T|\cdot |K_T:P_S\cap K_T|\cdot |P_S\cap K_T:K_S|\cdot |K_S|$$ and similarly the order of $H$ as
      $$|H| = |H:Q_T|\cdot |Q_T:Q_SL_T|\cdot |Q_SL_T:L_T|\cdot |L_T:Q_S\cap L_T|\cdot |Q_S\cap L_T:L_S|\cdot |L_S|$$
      Then with the above isomorphisms, we have
      \begin{equation*}
          \frac{|G|}{|H|} = \frac{|G:P_T|\cdot|K_T:P_S\cap K_T|\cdot|K_S|}{|H:Q_T|\cdot|L_T:Q_S\cap L_T|\cdot|L_S|}
      \end{equation*}
      Since by the assumption $P_T = G$, $K_S = 1$ and $K_T\le P_S$ we deduce that $\frac{|G|}{|H|}\le 1$.
\end{proof}
\end{nothing}


\begin{nothing}{\bf Composition product of sections.} \label{composition}

\end{nothing}
 Given finite groups $G$, $H$, $K$ and Galois morphisms $\Big(\frac{G\times H}{X\xrightarrow{g} Y}\Big)$ and 
 $\Big(\frac{H\times K}{X'\xrightarrow{f} Y'}\Big)$, we define their composition as the morphism
 \[ 
 \left( \frac{G\times K}{X\times_H X'\xrightarrow{g\times_H f} Y\times_H Y'}\right) 
\] 
Here $X\times_H X'$ and $Y\times_H Y'$ are the usual amalgamated products of bisets and 
$(g\times _H f)(x,_{H} x') = \big (g(x),_H f(x')\big )$.  It is a Galois Morphism of $G\times K$-sets by Proposition 9.13 of \cite{slice ring}. Moreover the above composition induces a bilinear associative product
\[
\Gamma(G, H)\times \Gamma(H, K) \to \Gamma(G, K).
\]
 The following corollary gives an explicit formula for the composition of two bases elements of the section Burnside ring. 
 
\begin{coro}{(Mackey Formula)}\label{Mackey Formula}
Let $(T,S)\in \Sigma_{G\times H}$ and $(V,U)\in \Sigma_{H\times K}$. There exists an isomorphism of Galois morphisms
 \begin{equation*}
     \Big(\frac{G\times H}{S\unlhd T}\Big)\times_{H} \Big(\frac{H\times K}{U\unlhd V}\Big)\cong \bigsqcup_{t\in p_2(S)\setminus H /p_1(U)}\Big(\frac{G\times K}{S*{^{(t,1)}U}\unlhd T*{^{(t,1)}V}}\Big)
 \end{equation*}
 where $S*U := \{(g,k)\in G\times K\, |$ there exists $h\in H$ with $(g,h)\in S$ and $(h,k)\in U\}.$
 \end{coro}

\begin{proof}
  This follows directly from the Mackey formula for bisets together with 
 \cite[Lemma 3.3]{slice ring}.
\end{proof}





\begin{notation}
Let $f:X\rightarrow Y$ be a Galois morphism in $\Gamma (G\times H)$. We define the opposite $f^{op}\in \Gamma (H\times G)$ of $f$ as the Galois morphism $f^{op}:X^{op}\rightarrow Y^{op}$ where $f^{op}(x) = f(x)$ for any $x\in X$. Here the opposite biset $X^{op}$ is the $(H,G)$-biset equal to $X$ as a set with the action defined in the following way. For all $h\in H$, $x\in X$, $g\in G$, $h\cdot x\cdot g\; (in\; X^{op}) = g^{-1}\cdot x\cdot h^{-1}(in\; X)$. In particular, if $f = \Big ( \frac{G\times H}{S\unlhd T}\Big )$ then $f^{op} \cong \Big ( \frac{H\times G}{S^{op}\unlhd T^{op}}\Big )$, where $S^{op} := \{(h,g)\in H\times G\mid (g,h)\in S\}$.
\end{notation}

\begin{notation}
    Let $G$ and $H$ be finite groups. For any subgroups $A\le B\le H$ and any group homomorphism $\alpha :B\rightarrow G$, we set
    \begin{equation*}
        _{\alpha} \Delta(A) := \{(\alpha (a),a)\mid a\in A\}\le G\times H.
    \end{equation*}
    For any subgroups $C\le D\le G$ and any group homomorphism $\alpha :D\rightarrow H$, we set
    \begin{equation*}
        \Delta _{\alpha} (C) := \{(c,\alpha (c))\mid c\in C\}\le G\times H.
    \end{equation*}
    If $\alpha$ is the inclusion map of a subgroup, we write $\Delta (A)$ and $\Delta (C)$ respectively. 
\end{notation}

\begin{nothing}{\bf Basic bisets.} \label{operations}
Let $G$ be a finite group and $X$ be a $G$-set. The identity morphism $X\to X$ is a Galois morphism and induces a homomorphism from the Burnside ring $B(G)$ to $\Gamma(G)$. More explicitly if $X = G/H$ for some subgroup $H$ of $G$, then its image in $\Gamma(G)$ is $[H, H]_G$. In particular the basic bisets of induction, restriction, inflation and deflation have images in $\Gamma(G, G)$ as listed below: Let $H\le G$ and $N\unlhd G$ and $\pi:G\to G/N$ be the natural homomorphism.
    \begin{equation*}
    \Ind_H^G := \Big(\frac{G\times H}{\Delta(H)\unlhd \Delta(H)}\Big)\in \Gamma(G, H), \quad 
        \Res_H^G := \Big(\frac{H\times G}{\Delta(H)\unlhd \Delta(H)}\Big)\in \Gamma(H, G)    
\end{equation*}

\begin{equation*}
    \Inf_{G/N}^G := \Big(\frac{G\times G/N}{\Delta_{\pi}(G)\unlhd \Delta_{\pi}(G)}\Big)\in \Gamma(G, G/N),\quad
        \Def_{G/N}^G := \Big(\frac{G/N\times G}{_{\pi}\Delta(G)\unlhd {_{\pi}\Delta(G)}}\Big)\in \Gamma(G/N, G)
\end{equation*}
With these definitions we may decompose any section into a product as follows. The statement is similar to 
\cite[Proposition 2.8]{fibered biset} and the proof is similar to the proof of \cite[Lemma 2.3.26]{biset functor}. We leave the straightforward calculations to the reader.
\end{nothing}

\begin{prop}\label{decomp galois}
Let G and H be finite groups and $(T, S)$ be a section of $G\times H$ with $l(T,S) = (P_T,K_T,P_S,K_S)$ and $r(T,S)=(Q_T,L_T,Q_S,L_S)$. Define $\bar{G} := {P_T}/{K_S}$, $\bar{H} := {Q_T}/{L_S}$, $\bar{S} := can(S/{({K_S}\times{L_S})})\leq \bar{G}\times \bar{H}$ and $\bar{T} := can(T/{({K_S}\times{L_S})})\leq \bar{G}\times \bar{H}$ where 
 \begin{center}
     $can:({P_T}\times {Q_T})/({K_S}\times {L_S})\rightarrow \bar G\times\bar H$
 \end{center}
is the canonical homomorphism.
Then
    \begin{equation*}
        \Big(\frac{G\times H}{S\unlhd T}\Big)\cong \Ind_{P_T}^G\times_{P_T}    \Inf_{\bar G}^{P_T}\times_{\bar G} \left(\frac{\bar G\times \bar H}{\bar S\unlhd \bar T}\right) \times_{\bar H} \Def_{\bar H}^{Q_T}\times_{Q_T} \Res_{Q_T}^{H} 
    \end{equation*}
 Furthermore we have 
 \begin{center}
     $k_1(\bar{S}) = 1 $, $k_2(\bar{S}) = 1 $, and $p_1(\bar{T}) = \bar{G}$, $p_2(\bar{T}) = \bar{H}$.
 \end{center}
\end{prop}



\section{The Green biset functor $\Gamma$}
The functor $\Gamma$ associating $\Gamma(G)$ to the finite group $G$ has a Green biset functor structure \cite[Theorem 10.7]{slice ring}. Explicitly, we make $\Gamma$ a biset functor as follows. Let $G$ and $H$ be finite groups, $U$ be a $(G, H)$-biset and 
$f: X\to Y$ be a Galois morphism of $H$-sets. Then the morphism 
\[
(f: X\to Y) \mapsto (U\times_H f: U\times_H X\to U\times_H Y)
\] 
induces a functor from $\HmorGal$ to $\GmorGal$ and hence an abelian group homomorphism from $\Gamma(H)$ to $\Gamma(G)$. Together with the product 
\[
\Gamma(G)\times \Gamma(G')\to \Gamma(G\times G'), (X\xrightarrow{f} Y, X'\xrightarrow{f'} Y')\mapsto ((X\times X') \xrightarrow{f\times f'} (Y\times Y'))
\]
the functor $G\mapsto \Gamma(G)$ becomes a Green biset functor.

As in Section 2.3, we associate the category $\mathcal P_{\Gamma}$ to the Green biset functor $\Gamma$ where the objects are all finite groups, the morphisms from $H$ to $G$ are given by $\Gamma(G\times H)$ and the composition is defined as follows. Given $\alpha\in\Gamma(H\times K)$ and $\beta\in\Gamma(G\times H)$, we have
\[
\beta\circ \alpha = \Gamma(\Def^{G\times\Delta(H)\times K}_{G\times K}\Res^{G\times H\times H\times K}_{G\times \Delta(H) \times K})(\beta\times \alpha).
\]
The next proposition shows that this composition is induced by the composition product we defined in the previous section.
\begin{prop}
Let  $\alpha  := \big[\frac{H\times K}{U\unlhd V}\big]$ and $\beta := \big[\frac{G\times H}{S\unlhd T}\big]$. Then
\[
\beta\circ \alpha = \beta \times_H\alpha.
\]
\end{prop}

\begin{proof}
First note that by definition $\beta \times \alpha = \big [\frac{G\times H\times H\times K}{S\times U\unlhd T\times V} \big]$. Now by  \cite[Theorem 10.6]{slice ring},
    \[
        \beta \circ \alpha =  \Def _{H\times K}^{B} \Res_B^A (\beta \times \alpha) = \Def_{H\times K}^B \big(\sum_{a\in [B\setminus A/ S\times U]}\big \langle(B\cap {^a}(T\times V)),(B\cap {^a}(S\times U))\big \rangle_B \big) = 
    \]
    \[
        \sum_{a\in [B\setminus A/S\times U]}\Big \langle \frac{(B\cap {^a}(T\times V))(1\times \Delta(G)\times 1)}{1\times \Delta(G) \times 1},\frac{(B\cap {^a}(S\times U))(1\times \Delta(G)\times 1)}{1\times \Delta(G) \times 1} \Big \rangle_{H\times K} 
    \]
where $A := G\times H\times H\times K $ and $B := G\times \Delta(H)\times K$.
Observe that $[B\setminus A/ (S\times U)] = [p_2(S)\setminus G/p_1(U)]$ via the bijection $Ba(S\times U)\mapsto p_2(S)g^{-1}\bar{g}p_1(U)$ for $a = (h,g,\bar{g},k)$.
Moreover we have a group isomorphism $\frac{(B\cap (T\times V))(1\times \Delta(G)\times 1)}{1\times \Delta(G) \times 1} \cong T*V$ given by the homomorphism $(g,1,1,k)[1\times \Delta (G) \times 1] \mapsto (g,k)$.

Thus we obtain
    \[
        \beta \circ \alpha = \sum _{g\in [p_2(S)\setminus G/p_1(U)]} \langle T*{^{(g,1)}}V,S*{^{(g,1)}}U \rangle_{H\times K}
    \]
which is equal to $\beta\times_H\alpha$ by the Mackey formula given in the previous section.
\end{proof}

Let $k$ be a commutative ring. We denote by $\mathcal P_\Gamma^k$ the category obtained from $\mathcal P_\Gamma$ by extending the coefficients of the morphisms to $k$, that is, by letting the morphisms from $H$ to $G$ to be $k\Gamma(G\times H)$ and extending the composition linearly. Then a $k$-linear functor $\mathcal P_\Gamma^k\to {}_k$mod is called a \emph{section biset functor (over $k$)}. We denote the functor category of all section biset functors by $\mathcal F = \mathcal F_\Gamma^k$.

As explained in Section 2, simple section biset functors are determined by the pairs $(G, V)$ where $G$ is a finite group and 
$V$ is a simple $E_G$-module. To obtain a parametrizing set for simple section biset functors, we study the structure of the endomorphism algebra $E_G$. Our results for the rest of the paper follow \cite{fibered biset} closely. We adapt the techniques from fibered biset functors to show that the section biset functors has a similar theory.

\section{The ring $E_G$ of endomorphisms}

In this section we investigate structure of the endomorphism ring $\Gamma(G\times G)$ in order to understand its module category. Throughout this section, let $k$ be a field of characteristic zero. We write $E_G = E_G^k = k\otimes \Gamma(G\times G)$. The main result in this section is Theorem \ref{k-algiso} which states that the covering algebra $E_G^c$ defined below is isomorphic to a direct product of matrix rings over certain group algebras. This is a version of Theorem 6.2 in \cite{fibered biset}. Some preliminary proofs are almost identical to the original versions stated and proved in \cite{fibered biset}. Whenever possible, we leave the verifications of the proofs to the reader. We present certain proofs in the appendix.

\begin{nothing}{\bf The poset $\G$.} 
We write $\G$ for the set of all pairs $(K,P)$ of normal subgroups of $G$ such that $[K, P] = 1$. Note that a pair
$(K, P)$ of normal subgroups of $G$ is in $\G$ if and only if $K\le C_G(P)$ if and only if $P\le C_G(K)$. 
There is a poset structure on $\G$ given by $(K,P)\preceq (L,Q)$ if $K\le L$ and $P\ge Q$. The pair $(1, G)$ is the minimum element and the pair $(G, 1)$ is the maximum element of this poset. In particular joins exists in $\G$ and given 
$(K, P), (L, Q)\in \mathcal G_G$, one has $(K, P)\vee (L, Q) = (KL, P\cap Q)$. 
\end{nothing}

\begin{nothing}
To apply the results in Section \ref{app:ipots}, for each $(K, P)\in \G$, we define
\[
E_{(K, P)} = \left( \frac{G\times G}{\Delta(P)\unlhd \Delta_K(G)} \right)
\]
where $$\Delta_K(G) = (K\times 1)\Delta(G) =\{ (g, h)\in G\times G | h^{-1}g\in K\}.$$ Clearly $\Delta(P) \le \Delta_K(G)$ and normality follows since $[K, P] =1$. Alternatively, if 
\begin{center}
        $J_K^G := \Inf_{G/K}^G \times_{G/K} \Def_{G/K}^G \in B(G,G)$
    \end{center}
 and 
    \begin{center}
        $I_P^G := \Ind_{P}^G \times_{P} \Res_{P}^G \in B(G,G)$
    \end{center}
  then the map $I_P^G\to J_K^G$ is a Galois morphism and its image in $\Gamma(G\times G)$ is $[E_{(K, 
   P)}]$. The left and the right invariants of the section $(\Delta_K(G), \Delta(P))$ are both given by $(G, K, P, 1)$.
  
\end{nothing}

\begin{defi}
For each pair $(K,P)\in \G$ we set 
     \begin{equation*}
         e_{(K,P)} := \frac{1}{|G:P|}\big[ E_{(K,P)}\big]\in E_G.
     \end{equation*}
\end{defi}
The following proposition summarizes basic multiplication formulas for the elements $E_{(K, P)}$. The proofs follow easily by direct calculations using the Mackey product formula given in Corollary \ref{Mackey Formula}. We leave the justifications to the reader.

\begin{prop}
 \label{prop4.2}
 Let $G$ and $H$ be finite groups. If $(K,P),(K',P')\in \G$ and $(T,S)\in \Sigma_{G\times H}$ with $l(T,S) = (G,K',P',1)$ then
 
    \begin{enumerate}[(a)]
        \item $E_{(K,P)}\times_{G} (\frac{G\times H}{S\unlhd T})\cong |P\backslash G/P'|
 (\frac{G\times H}{U\unlhd V})$ with $l(V,U) = (G,KK',P\cap P',1)$,
    In particular one has
    \begin{equation*}
         E_{(K,P)}\times_{G} E_{(K',P')} \cong |P\backslash G/P'|E_{(KK',P \cap P')}.
    \end{equation*}
        \item Assume that  $(K,P)\preceq (K',P')$. Then
    \begin{equation*}
        E_{(K,P)}\times_{G}\Big(\frac{G\times H}{S\unlhd T}\Big) \cong |G:P|\Big(\frac{G\times H}{S\unlhd T}\Big)
    \end{equation*}
    In particular
    \begin{equation*}
        E_{(K,P)}\times_G E_{(K',P')}\cong |G:P|E_{(K',P')}\quad and\quad E_{(K,P)}\times_G E_{(K,P)}\cong |G:P|E_{(K,P)}.
    \end{equation*}
        \item Assume that $r(T,S) = (H,L,Q,1)$ for some $(L,Q)\in \mathcal{G}_H$. Then
        \begin{center}
            $(\frac{G\times H}{S\unlhd T})\times_{H} (\frac{G\times H}{S\unlhd T})^{op}\cong |H:Q| E_{(K',P')}$    
        \end{center}
    \end{enumerate}
\end{prop}

\begin{nothing} With this proposition the elements $e_{(K, P)}, (K, P)\in\G$ are idempotents in $E_G$ and we have 
\[
e_{(K, P)}\cdot e_{(L, Q)} = e_{(KL, P\cap Q)} \quad \mbox{\rm and}\quad e_{(1, G)} = 1_{E_G}.
\]
In particular the conditions in Section \ref{app:ipots} are satisfied. We have the commutative subalgebra
generated by $e_{(K, P)}$ as $(K, P)$ runs over $\G$. Also setting 
  \begin{equation}
     \label{fbyes}
         f_{(K,P)} := \sum_{(K,P)\preceq (L,Q)\in \G}\mu_{(K,P),(L,Q)}e_{(L,Q)}
     \end{equation}
     where $\mu_{?, ?}$ is the Mobius function of the poset $\G$, we get
    \begin{equation}
     \label{ebyfs}
          e_{(K,P)} = \sum_{(K,P)\preceq(L,Q)\in \mathcal{G}_G}f_{(L,Q)}   
     \end{equation}
and
\begin{equation}
 \label{allfsequal1}
     \sum_{(K,L)\in \mathcal{G}_G}f_{(K,L)} = e_{(1,G)} = 1\in E_G   
 \end{equation}
  and the following corollary of Proposition \ref{prop:ipot-ortho}.
\end{nothing}
\begin{prop} 
 \label{prop4.5}
For all $(K,P), (L,Q)\in \G$ we have 
 \begin{center}
     $e_{(K,P)}\cdot f_{(L,Q)} = f_{(L,Q)}\cdot e_{(K,P)} = \left\{\begin{array}{l}f_{(L,Q)},\quad \mbox{if} \quad (K,P)\preceq (L,Q),\\0, \qquad \quad \mbox{otherwise},\end{array}\right.$
 \end{center}
 and 
 \begin{center}
     $f_{(K,P)}\cdot f_{(L,Q)} = \left\{\begin{array}{l}f_{(K,P)},\quad \mbox{if} \quad (K,P) = (L,Q),\\0,\qquad \quad \mbox{otherwise}.\end{array}\right.$
 \end{center}
 \end{prop}

\begin{nothing}{\bf Linked pairs.}
Although the idempotents $f_{(K, P)}$ for $(K, P)\in \mathcal G_G$ are orthogonal, they are not necessarily central in $E_G$. Next we define an equivalence relation on $\mathcal G_G$ so that the class sums of $f_{(K, P)}$ becomes central idempotents in the covering algebra that we construct below. Our first definition is more general. 

     Let $G$ and $H$ be finite groups and $(K,P)\in \G$ and $(L,Q)\in \mathcal{G}_H$. We say that the quadruples $(G,K,P,1)$ and $(H,L,Q,1)$ are \textit{linked} if there exists $(T, S)\in \Sigma_{G\times H}$ with $l(T,S) = (G,K,P,1)$ and $r(T,S) = (H,L,Q,1)$. In this case, we write $(G,K,P,1)\underset{(T,S)}{\sim }(H,L,Q,1)$ or just $(G,K,P,1)\sim (H,L,Q,1)$. In the case $H = G$, and $(K,P), (L, Q)\in \mathcal{G}_{G}$ are said to be \textit{$G$-linked} if $(G,K,P,1)\sim (G, L, Q, 1)$. We write $(K,P)\sim_{G} (L, Q)$ or just $(K,P)\sim (L, Q)$. Being linked is clearly an equivalence relation and we write $\{K,P\}_{G}$ for the $G$-linkage class of $(K,P)$.

Note that when $(G,K,P,1)\underset{(T,S)}{\sim }(H,L,Q,1)$ then by Goursat Theorem for sections, the corresponding crossed modules $(P, G/K, i_P)$ and $(Q, H/L, i_Q)$ are isomorphic where $i_P$ is the restriction of the natural homomorphism $G\to G/K$ to $P$. The converse is also true as shown below.
\end{nothing}

\begin{nothing} Let $(K,P)\in \G$. There is an action of $G/K$ on $P$ via conjugation
since $P$ is normal in $G$ and $[K, P] = 1$. It is straightforward to check that $(P,G/K,i_P)$ becomes a crossed module where  $i_P$ is as above. With this notation we get the following for the converse of the linkage condition.
\end{nothing}

\begin{prop}
Let $G$ and $H$ be finite groups and $(K,P)\in \G$ and $(L,Q)\in \mathcal{G}_H$. Then the following are equivalent:
\begin{enumerate}[(a)]
    \item $(G, K, P, 1)\sim (H, L, Q, 1)$.
    \item $(P, G/K, i_P)$ isomorphic to $(Q, H/L, i_Q)$ as crossed modules.
\end{enumerate}
\end{prop}

\begin{proof}
  The first part, $(a)$ implies $(b)$, is Theorem \ref{thm:Goursat}(S7). For the converse, if $(\alpha , \beta ) : (Q,H/L,i_Q)\rightarrow (P,G/K,i_P)$ is an isomorphism of crossed modules, then the conditions (S3) - (S7) are satisfied and hence there exists a section $(T, S)$ with left and right invariants given by $(G, K, P, 1)$ and $(H, L, Q, 1)$, respectively. In particular $(G, K, P, 1)$ and $(H, L, Q, 1)$ are linked. 
\end{proof}

\begin{nothing}
The partial order on the set $\G$ induces a partial order on the set $\G/ \sim$ of linkage classes: Given $(K,P), (L,Q)\in \G$, we write $\{K,P\}_G\preceq\{L,Q\}_G$ if and only if there exists $(K',P')\in \{K,P\}_G$ and 
$(L',Q')\in \{L,Q\}_G$ with $(K',P')\preceq (L',Q')$. To see that the induced relation is transitive, we show that 
$\{ K, P\}_G \preceq\{L, Q\}_G$ if and only if for each $(K', P')\sim_G (K, P)$ there exists  $(L', Q')\sim (L, Q)$ such that $(K', P')\preceq (L', Q')$. Clearly the converse follows from the definition. To prove the forward implication, suppose $(K, P)\preceq (L, Q)$ and $(K', P')\underset{(T, S)}{\sim}(L, Q)$ and consider
\[
E_{(L, Q)}\cdot\Big(\frac{G\times G}{S\trianglelefteq T}\Big)=\Big(\frac{G\times G}{S'\trianglelefteq T'}\Big)
\]
By Proposition \ref{prop4.2}(a), we have $l_0(T', S') = (L, Q)$. Putting $r_0(T', S') = (L', Q')$, we get $(L', Q')\sim (L, Q)$ and $(K', P')\preceq (L', Q')$.
With this definition, we write
\[    e_{\{K,P\}_G} := \sum_{(K',P')\in \{K,P\}_G}e_{(K',P')}\in \Gamma (G,G)
\]
and
\begin{equation*}
    f_{\{K,P\}_G} := \sum_{(K',P')\in \{K,P\}_G}f_{(K',P')}\in \Gamma (G,G)
\end{equation*}
for $(K, P)\in \G$.
The conditions of Proposition \ref{prop:ipot-ortho} are also satisfied by the elements $e_{\{K, P\}_G}$ and 
$f_{\{K, P\}_G}$ as $(K, P)$ runs over linkage classes in $\mathcal G_G$. Hence we obtain the following corollary.
\end{nothing}

\begin{prop} \label{prop6.4}
Let $(K,P), (L,Q)\in \mathcal{G}_G$. Then
\begin{equation*}
    e_{\{K,P\}_{G}} f_{\{L,Q\}_{G}} = f_{\{L,Q\}_{G}} e_{\{K,P\}_{G}} = 0 \quad \mbox{unless} \quad \{K,P\}_G\preceq\{L,Q\}_G
\end{equation*}

\begin{equation*}
    e_{\{K,P\}_{G}} f_{\{K,P\}_{G}} = f_{\{K,P\}_{G}} e_{\{K,P\}_{G}} = f_{\{K,P\}_{G}}\quad \mbox{and} 
\end{equation*}

\begin{equation*}
    f_{\{K,P\}_{G}} f_{\{L,Q\}_{G}} = \left\{\begin{array}{l}f_{\{K,P\}_{G}},\quad \mbox{if} \quad \{K,P\}_{G} = \{L,Q\}_{G},\\0, \qquad \quad \mbox{otherwise}.\end{array}\right.
\end{equation*}
\end{prop}

\begin{nothing}{\bf Covering algebra.} As we have shown in Proposition \ref{decomp galois}, for any section $(T, S)$ in $G\times G$, there is a factorization of $(\frac{G\times G}{S\unlhd T})$ over $p_1(T)/k_1(S)$ and $p_2(T)/k_2(S)$. Hence while determining the essential algebra $\bar E_G$, we can work with the smaller subalgebra where this type of factorization is trivial. 
Following \cite{fibered biset}, we call the pair $(T,S)\in \Sigma_{G\times H}$ {\bf covering} if $p_1(T) = G, p_2(T) = H$ and $k_1(S) = k_2(S) =1$, and denote the set of all such pairs by $\Sigma_{G\times H}^c$.  Notice that for a covering pair $(T, S)$, we have $l_0(T,S) \in \G$ and $r_0(T,S)\in \mathcal{G}_H$. It is clear that the 
elements $[\frac{G\times G}{S\unlhd T}]$ as $(T, S)$ runs over all covering pairs generates a subalgebra of 
$E_G$. We denote it by $E_G^c$. This is a covering algebra since by Proposition \ref{decomp galois} any basis element outside of $E_G^c$ factors through a group of smaller order.

\begin{nothing} Next we aim to determine the covering algebra $E_G^c$ more explicitly so that its intersection with the ideal $I_G$ is easier to unearth. For this aim, we first decompose $E_G^c$ as a product of matrix algebras. Note that for any $(K, P)\in \G$, the idempotents $e_{(K, P)}$ and $f_{(K, P)}$ are contained in $E_G^c$. The following theorem introduces certain groups contained in $E_G^c$ together with special bisets. Later, as Theorem \ref{k-algiso}, we show that the covering algebra is a direct sum of group algebras over these groups. Some parts of this theorem are given as separate results in \cite{fibered biset}.
\end{nothing}

\begin{theorem}\label{thm:Gammas}
Let $G, H$ be finite groups and let $(K,P)\in \mathcal{G}_G$ and $(L, Q)\in \mathcal G_H$ be linked pairs. 
Also let
\[ \Gamma_{(G,K,P)} = \left\{ \frac{1}{|G:P|}\left[ \frac{G\times G}{S\unlhd T} \right] \mid l(T,S) = (G,K,P,1) = r(T,S)\right\}
\]
and 
\[ {}_{(G, K, P)}\Gamma_{(H, L, Q)} = \left\{ \left[ \frac{G\times H}{U\unlhd V} \right] \mid l(V,U) = (G,K,P,1),  r(V, U) =(H, L, Q, 1)\right\}.
\]
 Then
\begin{enumerate}[(a)]
\item The set $\Gamma_{(G,K,P)}$ is a finite group under the multiplication induced by the product in $E_G$. The identity is $e_{(K, P)}$ and inverses are given by taking opposites.
\item The set ${_{(G, K, P)}}\Gamma_{(H, L, Q)} $ is a $(\Gamma_{(G,K,P)},\Gamma_{(H,L,Q)})$-biset which is both left and right transitive and left and right free.
\item Any $\gamma\in {_{(G, K, P)}}\Gamma_{(H, L, Q)}$ induces a group isomorphism
\begin{equation*}
    \gamma : \Gamma_{(H,L,Q)} \xrightarrow{\sim} \Gamma_{(G,K,P)}.
\end{equation*}
\item The functor 
\[ k[_{(G,K,P)}\Gamma_{(H,L,Q)}]\otimes_{k\Gamma_{(H,L,Q)}} -: {}_{k\Gamma_{(H,L,Q)}}\rm{mod} \to {}_{k\Gamma_{(G,K,P)}}\rm{mod}\]
is an equivalence of categories. It induces a canonical bijection \begin{equation}
\label{canbijofirr}
    \Irr (k\Gamma_{(H,L,Q)}) \xrightarrow{\sim} \Irr (k\Gamma_{(G,K,P)}).
\end{equation}
\item There is an isomorphism of groups
\begin{equation*}
    \Gamma_{(G,K,P)} \cong \Out(P, G/K, i_P)
\end{equation*}
where the right hand side is the group of outer automorphisms of the crossed module $(P, G/K, i_P)$.
\end{enumerate}
\end{theorem}

\begin{proof} 
All except the last claim are easy justifications, we leave the details to the reader. We only prove the last part by constructing an isomorphism. Define
\[
\Theta: \Aut(P, G/K, i_P)\to \Gamma_{(G, K, P)}
\] 
associating $(\alpha, \beta)$ to the element in $\Gamma_{(G, K, P)}$ corresponding to the section $(T, S)$ where
\[
T = \{ (x,y)\mid xK = \beta (yK)\, \mbox{for} \, x,y\in G \}
\]
and
\[
S = \{ (\alpha(p),p)\mid p\in P \}
\]
Note that $(T, S)$ is indeed a section by Goursat's theorem for sections. It is easy to check that $\Theta$ is a group homomorphism. Moreover if $(T, S)$ is a section of $G\times G$ with $l_0(T, S) = (K, P) = r_0(T, S)$ then by Goursat's Theorem for sections, we obtain an automorphism of the crossed module $(P, G/K, \iota_P)$. It remains to show that the kernel of $\Theta$ is the group of inner automorphisms. 
 \end{proof}

As a preparation for the proof of Theorem \ref{k-algiso}, we need the following technical results. Statements and proofs are very similar to the ones proved in \cite[Section 6]{fibered biset}. We state our versions in this section without proofs and collect all the proofs in the appendix. The following proposition is the version of Lemma 6.3 and its corollary from \cite{fibered biset}. 

\begin{prop}\label{prop:central} With the above notation
\begin{enumerate}[(a)]
\item Let $(T, S)$ be covering and $(K, P), (L, Q)\in \G$. If $f_{\{K, P\}_G}[\frac{G\times G}{S\unlhd T}]f_{\{L, Q\}_G}$ is non zero, then $\{ K, P\}_G = \{L, Q\}_G$.
\item The set $\{ f_{\{K, P\}_G}\mid \{K, P\}_G\in \G / \sim\}$ is a set of mutually orthogonal central idempotents of $E_G^c$ and their sum is 1.
\item The covering algebra $E_G^c$ decomposes into its two-sided ideals as
    \begin{equation}
    \label{decomp}
        E_G^c = \bigoplus_{\{K, P\}_G\in \G/\sim} f_{\{K, P\}_G}E_G^c.
    \end{equation}

\end{enumerate}
\end{prop}

\begin{nothing}
We can also decompose $E_G^c$ into $k$-submodules via linkage classes. For any 
$\{K, P\}_G\in\G/\sim$, let $E_G^{c, \{K, P\}}$ be the submodule of $E_G^c$ generated by all 
$\big[\frac{G\times G}{S\unlhd T}\big]$ with $(T, S)\in \Sigma _{G\times G}^c$ and $l_0(T, S)\sim_G (K, P)$ (or equivalently $r_0(T, S)\sim_G (K, P)$). Then
\begin{equation}
\label{moduldecomp}
    E_G^c = \bigoplus_{\{K, P\}_G\in \G/\sim} E_G^{c, \{K, P\}}
\end{equation}

This decomposition is related to the ideal decomposition in the following way. The following version of \cite[Lemma 6.6]{fibered biset} hides some parts of the result that are needed for the proof. The proof is presented in Appendix \ref{app:proofs}.
\end{nothing}
\end{nothing}

\begin{lem}\label{lem:w's}
Let $(K, P)\in\mathcal G_G$. 
\begin{enumerate}[(a)]
\item The following equality holds.
\[
\bigoplus_{\{K,P\}_G \preceq \{L,Q\}_G\in \G/\sim}E_G^c f_{\{L,Q\}_G} = \bigoplus_{\{K,P\}_G \preceq \{L,Q\}_G\in \G/\sim}E_G^{c,\{L,Q\}}
\]
\item The projection 
\[
\omega: E_G^{c, \{K, P\}} \to E_G^c f_{\{K, P\}_G},\; b\mapsto bf_{\{K, P\}_G}
\]
with respect to the ideal decomposition of $E_G^c$ is an isomorphism of $k$-modules. Its inverse is given by the projection with respect to the submodule decomposition of $E_G^c$.
\item If $\{K, P\}_G = \{(K_1, P_1), (K_2, P_2), \ldots, (K_n, P_n)\}$ then $\omega$ becomes the direct sum, over $i$ and $j$, of the $k$-module isomorphisms 
\[
\omega_{ij}: k[{}_{(G, K_i, P_i)} \Gamma_{(G, K_j, P_j)}] \to f_{(K_i, P_i)}E_G^cf_{(K_j, P_j)}
\]
given by $b_{ij} \mapsto f_{(K_i, P_i)}bf_{(K_j, P_j)}$.
\item The isomorphism $\omega_{ii}$ as defined above induces an isomorphism of $k$-algebras
\[
k[\Gamma_{(G, K, P)}] \to f_{(K, P)}E_G^cf_{(K, P)}
\]
given by $a\mapsto f_{(K, P)}af_{(K, P)}$.
\end{enumerate}
\end{lem}
Now we are ready to state the main result of this section. The proof is given in Appendix \ref{app:proofs}.

\begin{theorem}\label{k-algiso}
There exists a $k$-algebra isomorphism

\begin{center}
    $E_G^c\xrightarrow{\sim} \underset{\{K,P\}_G\in \G/\sim}{\bigoplus} $ Mat $_{|\{K,P\}_G|}k\Gamma_{(G,K,P)} $
\end{center}
with the following property: 

\noindent For every $\{K,P\}_G = \{(K_1,P_1),\cdots , (K_n,P_n)\}\in \G/\sim$, the isomorphism maps $f_{(K_i,P_i)}\in E_G^c$ to the idempotent matrix $e_i = \mbox{diag}(0,\cdots ,0, 1,0,\cdots , 0)\in \mbox{Mat}_{|\{K,P\}_G|}(k\Gamma_{(G,K,P)})$, for $i = 1, \cdots , |\{K,P\}_G|$, in the $\{K,P\}_G$-component. 
\end{theorem}

\section{The essential algebra ${\bar E}_G$}
Let $k$ be a field of characteristic zero, and fix a finite group $G$. In this section we determine the essential algebra ${\bar E}_G$ and its simple modules. Once more the results in this section are similar to the results in \cite[Section 8]{fibered biset}. 

\begin{nothing} As in the case of fibered biset functors, the 
essential algebra is isomorphic to a subalgebra of the covering algebra. Generators of this subalgebra are characterized as follows. Let $(K,P)\in \G$. We call $(K, P)$ a {\bf reduced pair} if $e_{K, P}$ is not contained in the
essential ideal $I_G$. We denote the subset of $\G$ consisting of the reduced pairs by $\R_G = \R_k(G)$. It is easy to prove that being reduced is compatible with being linked, that is, if $(K,P),(K',P')\in \mathcal{G}_G$ are $G$-linked then $(K,P)$ is reduced if and only if $(K',P')$ is reduced. (cf. \cite[Notation 8.1]{fibered biset}) We write $\widetilde{\mathcal R}_G$ for the set of linkage classes of reduced pairs for $G$. Note also that $\R_G$ is a lower set, that is, if $(K,P)$ is reduced and
$(K', P')\preceq (K, P)$, then $(K', P')$ is also reduced.

The following lemma determines a basis for the essential ideal. Its proof is similar to the proof of \cite[Lemma 8.2]{fibered biset} and our version is given in Appendix \ref{app:proofs}.
\end{nothing}

\begin{lem}
\label{lemma8.1}
The ideal $I_G$ of $E_G$ is generated as a $k$-module by the standard basis elements $\big[\frac{G\times G}{S\unlhd T}\big]$ with $(T,S)\in \Sigma_{G\times G}$ satisfying 
\begin{enumerate}[($i$)]
    \item  $p_1(T) \neq G$ and $k_1 (S) \neq 1$ or
    \item $p_1(T) = G$, $k_1 (S) = 1$ and $l_0(T,S) \not \in \R_G$
\end{enumerate}
Equivalently, it is generated by the standard basis elements $\big[\frac{G\times G}{S\unlhd T}\big]$ with $(T,S)\in \Sigma_{G\times G}$ satisfying 
 \begin{enumerate}[($i^\prime$)]
        \item $p_2(T) \neq G$ and $k_2 (S) \neq 1$ or
        \item $p_2(T) = G$, $k_2 (S) = 1$ and $r_0(T,S) \not \in \R_G$
    \end{enumerate}
\end{lem}

Now we are ready to determine the intersection of the covering algebra and the ideal $I_G$. This also reveals the structure of the essential algebra. The proof is given in Appendix \ref{app:proofs}.

\begin{prop}
\label{prop9.5}

\begin{enumerate}[(a)]
    \item The algebra $E_G^c$ is a covering algebra for $E_G$, in the sense of Section 2.5.
       \item The equality
        \begin{equation}
         \label{redintcov}
             E_G^c\cap I_G = \bigoplus_{\begin{smallmatrix} \{K,P\}_G\in \G/\sim & \\ (K,P)\not \in \R_G  \end{smallmatrix}} f_{\{K,P\}_G}E_G^c
        \end{equation}
        holds.
    \item The canonical epimorphism $E_G\rightarrow \bar{E}_G$ maps the subalgebra 
    $$\bigoplus_{\{K,P\}_G\in \widetilde{\R}_G}f_{\{K,P\}_G}E_G^c$$ of $E_G^c$ isomorphically onto $\bar{E}_G$.
    \item For each $(K,P)\in \R_G$, the map 
        \begin{equation*}
            k\Gamma_{(G,K,P)}\rightarrow \bar{f}_{(K,P)}\bar{E}_G \bar{f}_{(K,P)}, \quad a\mapsto \bar{f}_{(K,P)}\bar{a} \bar{f}_{(K,P)}
        \end{equation*}
    is a $k$-algebra isomorphism.
    \item There is a bijective correspondence between the isomorphism classes of simple $\bar E_G$-modules and the set of triples $(K, P, [V])$ where $(K, P)$ runs over linkage classes of reduced pairs for $G$ and for a representative $(K, P)$ of the linkage class of $(K, P)$, $[V]$ runs over irreducible $k\Gamma_{(G, K, P)}$-modules. The correspondence is given by associating the triple $(K, P, [V])$ to the $\bar E_G$-module $\bar E_G\bar{f}_{(K, P)}\otimes_{k\Gamma_{(G, K, P)}} V$.
    \end{enumerate}
\end{prop}

With this result it is important to know more explicit results about reduced pairs. Unfortunately we do not have a complete description of these pairs. The following theorem includes some partial results.

\begin{theorem}\label{prop:reduced}
    Let $(K,P)\in \mathcal{G}_G$.
    \begin{enumerate}[(a)]
        \item The pair $(K, P)$ is reduced if $K\le P$.
        \item If $(K,P)$ is reduced, then for any non-trivial $N\unlhd G$ with $N\le K$, we have $P\cap N \neq 1$.
      \item The pair $(K, P)$ is not reduced if there is a group $H$ and a pair $(L, Q)\in \mathcal G_H$ such that $|H|<|G|$ and
      $(G, K, P, 1)\sim(H, L, Q, 1)$.
    \end{enumerate}{}
\end{theorem}

\begin{proof}\begin{enumerate}[(a)]
        \item Assume for contradiction that $(K,P)$ is not reduced. Hence $e_{(K,P)}\in I_G$, that is there exist a finite group $H$ with order strictly less than the order of $G$ and sections $(T,S)$ and $(V,U)$ in $G\times H$ and $H\times G$ respectively such that $e_{(K,P)}$ occurs as a summand in $\Big[ \frac{G\times H}{S\unlhd T}\Big] \underset{H}{\cdot}\Big[ \frac{H\times G}{U\unlhd V} \Big]$. Then by the Mackey formula there is a section of $H\times G$, say $(V',U')$, such that $(S*U',T*V') = (\Delta_K(G),\Delta(P))$. Thus it follows that $l(T,S) = (G,K',P',1)$ for some $(K',P')\preceq (K,P)$. Now by Proposition \ref{prop4.2}(b), we have $e_{(K,P)}\underset{G}{\cdot} \Big[\frac{G\times H}{S\unlhd T} \Big] = \Big[\frac{G\times H}{S'\unlhd T'}\Big]$ with $(T',S') \in \Sigma_{G\times H}$ satisfying $l(T',S') = (G,K,P,1)$. Finally Proposition \ref{sections prop}(c) applied to $(T',S')$ implies that $|G|\le |H|$, a contradiction. 
        
        \item Let $(K,P)\in \R_G.$ Assume, for contradiction, that there exist $N\unlhd G$ with $N\le K$ such that $P\cap N = 1$. Let $x =\Big [ \Inf _{G/K}^G \Iso(\eta_{T})  \Def _{\frac{G/N}{K/N}}^{G/N}, \Ind_{P}^G\Iso(\eta_{S}) \Res ^{G/N}_{PN/N}\Big]$ and\\
        $y = \Big [ \Inf _{\frac{G/N}{K/N}}^{G/N} \Iso(\eta_{T}^{-1}) \Def _{G/K}^{G}, \Ind^{G/N}_{PN/N}\Iso(\eta_{S}^{-1})\Res_{P}^{G}\Big]$, with the canonical isomorphisms\\ $\eta_T : \frac{G/N}{K/N} \xrightarrow{\sim} G/K$ and $\eta_S : PN/N \xrightarrow{\sim} P$. Then $e_{(K,P)} = \frac{1}{|G:P|}x\cdot _{G/N} \frac{1}{|G:PN|} y \in I_G$, which is a contradiction since $|G/N|< |G|$.
    \item This part is straightforward by the definition of linkage.
       \end{enumerate}
\end{proof}
\begin{coro}
The pair $(K, P)\in \mathcal G_G$ is reduced in each of the following cases:
\begin{enumerate}[(a)]
        \item $K = 1$.
        \item $P = G$
\end{enumerate}
In particular the essential algebra $\bar E_G$ is non-zero for any finite group $G$.
\end{coro}

\begin{coro} The pair $(K, P)\in \G$ is not reduced in each of the following cases:
\begin{enumerate}[(a)]
        \item $P < K$.
        \item $PK = G$ and $K\not \le P$.
\end{enumerate}
\end{coro}
\begin{proof} In both cases it is sufficient to find a triple $(H, L, Q)$ with $|H|<|G|$ which is linked to $(G, K, P)$.
\begin{enumerate}[(a)]
 \item Let $G/K$ act on $P$ by conjugation and construct the semi-direct product $\bar{G} = P\rtimes G/K$. Write $\bar{P}$ for the image (under $\alpha^{-1}: x\mapsto 
(x, 1)$) of $P$ in $\bar{G}$ so that we have an isomorphism $\beta : \bar{G}/\bar{P} \xrightarrow{\sim} G/K$.
    
We define the subgroups $S = \{ (x, (x,1)) : x\in P \}$ and $T=\{ (g, (x, yK)): gK = \beta((x, yK)\bar P)\}$ of
$G\times \bar G$ and claim that $S\unlhd T$. Indeed we clearly have $S\le T$ and the normality follows by direct calculations. Also note that $P < K$ and 
 $[K, P] = 1$, hence $P$ is abelian. The Goursat correspondents of $T$ and $S$ are $(G, K, \beta, \bar P, \bar G)$ and $(P, 1, \alpha, 1, \bar P)$ respectively. Hence the quadruples $(G, K, P, 1)$ and $(\bar G, \bar P, \bar P, 1)$ are linked, as required.
 
 \item Let $S = \Delta(P)\le G\times P$ and $T = \{(g, h)\in G\times P | g^{-1}h\in K\}$. Clearly $S$ is a normal subgroup of $T$. Since $G = PK$, the pair $(T, S)$ is covering and hence $(G, K, P)$ is linked to $(P, K\cap P, P)$.
 Note that $K\not \le P$ is also necessary since otherwise $G = PK$ implies $G=P$ and the factorization is not over a group of smaller order.
    \end{enumerate}
\end{proof}

\begin{nothing}\label{sgbijirr}
 Finally we parameterize simple modules over $\bar{E}_G$.  Let
    \begin{equation*}
       \mathcal S_G = \mathcal S_k(G) := \{((K,P),[V])\mid (K,P)\in \R_G, [V]\in \Irr(k\Gamma_{(G,K,P)})\}.
    \end{equation*}
Then define an equivalence relation on $\mathcal S_G$ by
$        ((K,P),[V]) \sim ((K',P'),[V']) $ if $(K,P)$ and $(K',P')$ are $G$-linked and the canonical bijection $\mbox{\rm Irr}(k\Gamma_{(G,K',P')})\xrightarrow{\sim} \mbox{\rm Irr}(k\Gamma_{(G,K,P)})$ from Theorem \ref{thm:Gammas} maps $[V']$ to $[V]$.
 Also let $\Tilde{\mathcal S}_G$ be a set of representatives of the equivalence classes of $\mathcal S_G$, that is,
    \begin{equation*}
        \Tilde{\mathcal S}_G := \{((K,P),[V])\mid (K,P)\in \Tilde{\mathcal{R}}_G, [V]\in \mbox{\rm Irr}(k\Gamma_{(G,K,P)})\}.
    \end{equation*}

By the canonical isomorphism from Proposition \ref{prop9.5}(c), we can view each simple $k\Gamma_{(G,K,P)}$-module as a simple $\bar{f}_{(K,P)}\bar{E}_G \bar{f}_{(K,P)}$-module, and we can view $\bar{E}_G \bar{f}_{(K,P)}$ as $(\bar{E}_G,k\Gamma_{(G,K,P)})$-bimodule. Hence $\bar{E}_G \bar{f}_{(K,P)}\otimes_{k\Gamma_{(G,K,P)}} V$ is a simple $\bar E_G$-module for each simple $k\Gamma_{(G, K, P)}$-module. This induces a bijection between the set $\tilde{\mathcal S}_G$ and Irr$(\bar{E}_G)$.
\end{nothing}

\section{Simples functors over $k\Gamma$}
In this final section we parameterize simple section biset functors. Once more we follow \cite{fibered biset}. In the previous section we have classified simple modules over the endomorphism ring $E_G$ for any finite group $G$. This gives us the map from the set of seeds to the set of isomorphism classes of simple section biset functors. In this section we introduce an equivalence relation on seeds to make this map bijective. 

\begin{nothing}
By the parametrization in Section \ref{sgbijirr}, we write the set of all seeds for $k\Gamma$ as
\[
{\rm{Seeds}}(k\Gamma) = \{ (G, K, P, [V])| G\in Ob(\mathcal P_\Gamma), (K, P, [V])\in \tilde{\mathcal S}_G \}.
\]
Given a seed $(G, K, P, [V])$, we construct a simple $E_G$-module $\widetilde V$ by
\[
\widetilde V = \bar E_G\bar f_{(K, P)} \otimes_{k\Gamma_{(G, K, P)}} V.
\]
Hence as explained in Section \ref{sec:a-mods}, we associate a simple section functor $S_{G, \widetilde V} = S_{G, K, P, [V]}$ to the seed $(G, K, P, [V])$. Clearly replacing $V$ with another isomorphic copy of $V$ does not change the isomorphism type of the corresponding simple section biset functor. In particular we obtain a function
\[
\omega: {\rm Seeds}(k\Gamma) \to {\rm Irr}(k\Gamma)
\]
where Irr$(k\Gamma)$ denotes the set of isomorphism classes of simple section biset functors. Note that $\omega$ is surjective. Indeed if $S$ is a simple section biset functor, we let $G$ be a minimal group for $S$. Then $S(G)$ is a simple module for the essential algebra $\bar E_G$, and hence there is a triple $(K, P, [V])\in \tilde{\mathcal S}_G$ such that $S(G) \cong \bar E_G\bar f_{(K, P)} \otimes_{k\Gamma_{(G, K, P)}} V$. Therefore $S\cong S_{G, K, P, [V]}$.
\end{nothing}
\begin{nothing}
As in \cite{fibered biset}, the relation on Seeds$(k\Gamma)$ given by
\[
(G, K, P, [V]) \sim (H, L, Q, [W])
\] 
if and only if $(G, K, P, 1)\sim (H, L, Q, 1)$ and $V\cong k[{}_{(G, K, P)}\Gamma_{(H, L, Q)}]\otimes_{k\Gamma_{(H, L, Q)}} W$
is an equivalence relation. If these conditions are satisfied, we say that the quadruples $(G, K, P, [V])$ and $(H, L, Q, [W])$ are \emph{linked}. We claim that the function $\widetilde \omega$ induced by $\omega$ on the set of linkage classes Seeds$(k\Gamma)/\sim$ of seeds is bijective.
To prove this claim, we follow the techniques from \cite{fibered biset}. Ideas used in the proof below can be found in \cite[Section 9]{fibered biset}.
\end{nothing}
\begin{nothing}\label{sec:isomparts}
Let $M$ be a section biset functor and $G$ and $H$ be finite groups. We have decompositions
\[
M(G) = \bigoplus_{(K, P)\in\G} f_{(K, P)}M(G),\quad M(H) = \bigoplus_{(L, Q)\in \mathcal G_H} f_{(L, Q)}M(H).
\]
Although these decompositions may not be related to each other, in general, the terms corresponding to linked quadruples are isomorphic. Indeed suppose there are pairs $(K, P)\in\G$ and $(L, Q)\in \mathcal G_H$ such that the triples $(G, K, P)$ and $(H, L, Q)$ are linked so that
the set ${}_{(G, K, P)}\Gamma_{(H, L, Q)}$ is non-empty. Then, for each $x\in {}_{(G, K, P)}\Gamma_{(H, L, Q)}$, the map
\[
f_{(K, P)}M(G) \to f_{(L, Q)}M(H), m\mapsto |G:P|^{-1}x\cdot m
\]
is an isomorphism of $k$-modules with the inverse given by multiplication by $x^{\rm op}$ (cf. \cite[Lemma 9.3]{fibered biset}). This follows easily since $xx^{\rm op} = e_{(K, P)}$ and $x^{\rm op}x = e_{(L, Q)}$.
\end{nothing}
\begin{nothing}
Let $(H, L, Q, [W])$ be a seed and $S_{H, \widetilde W}$ be the corresponding simple section biset functor so that 
\[
S_{H, \widetilde W}(H) = \widetilde W  = \bar E_G \bar f_{(L, Q)} \otimes_{k\Gamma_{(H, L, Q)}} W.
\]
Note that since there is an isomorphism of $k$-algebras $k\Gamma_{(H, L, Q)}\cong f_{(L, Q)}E_H^c f_{(L, Q)}$, by Lemma \ref{lem:w's}, we may regard $f_{(L, Q)}S_{H, \widetilde W}(H)$ as a $k\Gamma_{(H, L, Q)}$-module, and as such, it is isomorphic to $W$. 

In order to determine other seeds which corresponds to the functor $S_{(H, \widetilde W)}$, we need to know the evaluations $S_{H, \widetilde W}(G)$ for groups with $|G| = |H|$. Hence let $G$ be a group of order $|H|$. Clearly there is a seed of the form $(G, K, P, [V])$ corresponding to $S_{H, \widetilde W}$ only if $S_{H, \widetilde W}(G)$ is non-zero, which we assume from now on. Then since $H$ is a minimal group for $S_{H, \widetilde W}$, $G$ is also minimal. In particular $S_{H, \widetilde W}(G)$ is annihilated by $I_G$, and hence it is a simple $\bar E_G$-module. 
Put $\widetilde V = S_{H, \widetilde W}(G)$. By Section \ref{sgbijirr}, there is a triple $(K, P, [V])$ such that $\widetilde V \cong
\bar E_G \bar f_{(K, P)} \otimes_{k\Gamma_{(G, K, P)}}V$. Hence $\bar f_{(K, P)}S_{H, \widetilde W}(G) \cong V$ and $S_{H, \widetilde W}\cong S_{G, \widetilde V}$.

Now by Lemma \ref{lem:technical}, we must have a surjective homomorphism
\[
k{}_{(G, K, P)}\Gamma_{(H, L, Q)} \otimes_{k\Gamma_{(H, L, Q)}} W\to V
\]
of $k\Gamma_{(G, K, P)}$-modules. In particular the set ${}_{(G, K, P)}\Gamma_{(H, L, Q)}$ is non-empty, which implies that the quadruples $(G, K, P, 1)$ and $(H, L, Q, 1)$ are linked. Moreover by Theorem \ref{thm:Gammas}, the $k\Gamma_{(G, K, P)}$-module $k{}_{(G, K, P)}\Gamma_{(H, L, Q)} \otimes_{k\Gamma_{(H, L, Q)}} W$ is simple. Hence we must have an isomorphism
\[
k{}_{(G, K, P)}\Gamma_{(H, L, Q)} \otimes_{k\Gamma_{(H, L, Q)}} W\cong V
\]
of $k\Gamma_{(G, K, P)}$-modules. As a result the seeds $(G, K, P, [V])$ and $(H, L, Q, [W])$ are linked, that is, seeds corresponding to the simple functor $S_{H, \widetilde W}$ are all linked to $(H, L, Q, [W])$.
\end{nothing}
\begin{nothing}
Finally we show that if $(G, K, P, [V])$ and $(H, L, Q, [W])$ are linked then $S_{G, \widetilde V}\cong S_{H, \widetilde W}$. First note that by Lemma \ref{lem:equalOrder}, the groups $G$ and $H$ are of the same order. As discussed above, the evaluation $S_{H, \widetilde W}(G)$
is a simple $\bar E_G$-module. Also since $(G, K, P, 1)\sim (H, L, Q, 1)$, we get $\bar f_{(K, P)}S_{H, \widetilde W}(G) \cong \bar f_{(L, Q)}S_{H, \widetilde W}(H)$ as $k$-modules, by Section \ref{sec:isomparts}. In particular $\bar f_{(K, P)}S_{H, \widetilde W}(G)$ is non-zero. It remains to show that there is an isomorphism $S_{H, \widetilde W}(G) \cong \widetilde V$ of $\bar E_G$-modules. We already know that
\[
V \cong k{}_{(G, K, P)}\Gamma_{(H, L, Q)}\otimes_{k\Gamma_{(H, L, Q)}}W. 
\]
Hence by Lemma \ref{lem:technical}, there is a homomorphism $V\to \bar f_{(K, P)}S_{H, \widetilde W}(G)$ of $k\Gamma_{(G, K, P)}$-modules. Now we have the following isomorphisms.
\begin{eqnarray*}
\{0\}\not = \Hom_{k\Gamma_{(G, K, P)}} (V, \bar f_{(K, P)}S_{H, \widetilde W}(G))
\cong \Hom_{k\Gamma_{(G, K, P)}} (V, \Hom_{\bar E_G}(\bar E_G\bar f_{(K, P)}, S_{H, \widetilde W}(G)))\\ 
\cong \Hom_{\bar E_G}(\bar E_G\bar f_{(K, P)}\otimes_{k\Gamma_{(G, K, P)}} V, S_{H, \widetilde W}(G))
\cong \Hom_{\bar E_G}(\widetilde V, S_{H, \widetilde W}(G)).
\end{eqnarray*}
In particular there is a non-zero homomorphism $\widetilde V\to S_{H, \widetilde W}(G)$. Since both of these modules are simple, they must be isomorphic. With this step, we have completed the proof of following theorem.
\end{nothing}
\begin{theorem}
There is a bijective correspondence between
\begin{enumerate}[(a)]
\item the set Irr$(k\Gamma)$ of isomorphism classes of simple section biset functors
\item the set Seeds$(k\Gamma)/\sim$ of linkage classes of quadruples $(G, K, P, [V])$.
\end{enumerate}
\end{theorem}
\begin{remarki}
Let $G$ and $H$ be non-isomorphic finite groups. For reduced pairs $(K,P)\in \R_G$ and $(L,Q)\in \R_H$ if there exists a section $(T,S)$ such that $(G, K, P, 1)\underset{(T,S)}{\sim} (H, L, Q, 1)$ then we can take $V := k{}_{(G,K,P)}\Gamma_{(H,L,Q)}\otimes_{k\Gamma_{(H,L,Q)}} W$ to be the irreducible  $k\Gamma_{(G,K,P)}$-module corresponding to $W\in \Irr(k\Gamma_{(H,L,Q)})$. Then $S_{G,\widetilde V} \cong S_{H,\widetilde W}$, which shows that there exists simple section biset functors with non-isomorphic minimal groups.\\

For example let $G$ be the group of quaternions, $\langle x,y \,|\, x^4 = 1,\, yxy^{-1} = x^{-1},\, x^{2} = y^{2}\rangle$ and $H := \langle a,b\,|\, a^4=b^2 = 1,\, bab^{-1} = a^{-1}\rangle$ the dihedral group of order 8. Consider the section $(T,S)$ of $G\times H$ given by
    \[
        T := \langle(x,a),(y,b)\rangle\, \rm{and}\, S :=\langle(x,a)\rangle
    \]
Clearly $S$ is a subgroup of $T$ and since $^{(y,b)}(x,a) = (x^{-1}, a^{-1})$ moreover $S$ is normal in $T$ and hence $(T,S)$ is a section of $G\times H$. The left and right invariant of $(T,S)$ are as follows
    \[
        l(T,S) = (G, \langle x^2\rangle, \langle x\rangle, 1) \quad r(T,S) = (H, \langle a^2\rangle, \langle a\rangle, 1).
    \]
So $(K,P) := l_0(T,S) = (\langle x^2\rangle,\langle x\rangle) \in \G$ and $(L,Q) := r_0(T,S) = (\langle a^2\rangle, \langle a\rangle)\in \mathcal{G}_H$ are linked pairs. Moreover by Theorem \ref{prop:reduced} (a) both $(K,P)$ and $(L,Q)$ are reduced pairs.

\end{remarki}

\appendix

\section{Proofs}\label{app:proofs}
\begin{nothing}\ {\bf Proof of \ref{prop:central}.} 
\begin{enumerate}[(a)]
    \item 
    By definition $f_{\{K,P\}}$ is a sum of $f_{(K',P')}$'s where $(K',P')$ runs through the linkage class of $\{K,P\}_G$. So for the product given in (a) being non-zero there should be at least a pair say $(K',P')$ in the linkage class $\{K,P\}_G$  and a pair $(L',Q')$ in the linkage class of $\{L,Q\}_G$ such that the following product is non-zero $f_{(K', P')}[\frac{G\times G}{S\unlhd T}]f_{(L', Q')}$. Moreover, $f_{(L',Q')}$ is the sum of $e_{(L'', Q'')}$'s where $(L'', Q'')$ runs through all poset elements of $\G$ which are bigger than $(L', Q')$. Then by the similar reasoning  $f_{(K', P')}[\frac{G\times G}{S\unlhd T}]e_{(L'', Q'')}$ must be non-zero at least for some $(L', Q') \preceq (L'', Q'')$. Hence we have $[\frac{G\times G}{S\unlhd T}]e_{(L'', Q'')} \not = 0$, and by the Proposition \ref{prop4.2} its in the form $[\frac{G\times G}{U\unlhd V}]$ for some $(V,U)\in \Sigma_G^c$. Also by the same proposition we know that the right middle invariant of section $(V,U)$ is bigger than $(L'', Q'')$. Let indicate by $(K'',P'')$ the left middle invariant of the section $(V,U)$. Then $f_{(K', P')} e_{(K'',P'')}[\frac{G\times G}{S\unlhd T}]e_{(L'', Q'')} =  f_{(K', P')}[\frac{G\times G}{S\unlhd T}]e_{(L'', Q'')} \not = 0$ and so  $f_{(K', P')} e_{(K'',P'')} \not = 0$, which implies $(K'',P'') \preceq (K',P')$ by Proposition \ref{prop4.5}. If we gather it all together,
    \begin{center}
        $ \{L,Q\}_G = \{L', Q'\}_G \preceq \{L'',Q''\}_G \preceq \{r_0(V,U)\}_G = \{l_0(V,U)\}_G = \{K'',P''\}_G \preceq \{K', P'\}_G = \{K, P\}_G.$
    \end{center}
Note that by the similar reasoning we also can derive $\{K,P\}_G \preceq \{L,Q\}_G$ and hence there must be an equality.
    
    \item By the Proposition \ref{prop6.4}, for $\{K,P\}_G \in \G / \sim $, the elements $f_{\{K,P\}_G}$ are the mutually orthogonal idempotents of the covering algebra. Also their sum is equal to the sum of all $f_{(K,P)}$ where $(K,P)$ run through all the poset elements of $\G$ and this sum is equal to $1$ by the equation (\ref{allfsequal1}). Hence for any $b\in E_G^c$
    \[ \sum _{\{K,P\}_G\in \G/\sim} f_{\{K,P\}_G} b = b = \sum _{\{K,P\}_G\in \G/\sim} b f_{\{K,P\}_G} 
    \]
    and by the previous case we obtain $f_{\{K,P\}_G}b = bf_{\{K,P\}_G}$ for every $\{K,P\}_G\in \G/\sim$.
    
    \item This is immediate from case $(b)$.
\end{enumerate}

\end{nothing}

\begin{nothing}\ {\bf Proof of \ref{lem:w's}.} 
        Using Proposition \ref{prop4.5} one can derive that the product $[\frac{G\times H}{S\unlhd T}] f_{(L,Q)}$ is zero for any section $(T,S)\in \Sigma_{G\times G}^c$ with $r_0(T,S) = (K,P)$ and $(K,P)\npreceq (L,Q)\in \G$. Similarly one can get its left-sided version too.
        \begin{enumerate}[(a)]
        \item To prove the equality first note that by the Proposition \ref{prop:central} the left hand side equals the annihilator of the set of all $f_{\{K',P'\}_G}$ with $\{K',P'\}_G \nsucceq \{K,P\}_G$. On the other hand by the observation above the right hand side also annihilates $f_{\{K',P'\}_G}$ for all $\{K',P'\}_G \nsucceq \{K,P\}_G$, which shows that the right hand side is contained in the left hand side. But right hand side obviously contains the left hand side since it contains $f_{\{L,Q\}_G}$ for every $\{L,Q\}_G\preceq \{K,P\}_G$.
        \item By the same way we proved previous case one can also show that 
        \[
            \bigoplus_{\{K,P\}_G \prec \{L,Q\}_G\in \G/\sim}E_G^c f_{\{L,Q\}_G} = \bigoplus_{\{K,P\}_G \prec \{L,Q\}_G\in \G/\sim}E_G^{c,\{L,Q\}}.
        \]
        Hence $E_G^{c,\{K,P\}}$ and $E_G^cf_{\{K,P\}_G}$ are both complements to the same submodule given above. So $\omega$ is an isomorphism of $k$-modules.
        \item By definition, and since $f_{\{K,P\}_G}$ is central in $E_G^c$ we have a direct sum decomposition
        \begin{equation}
            \label{decomp1}
            E_G^{c,\{K,P\}} = \bigoplus_{i,j =1}^n k[_{(G,K_i,P_i)}\Gamma _{(G,K_j,P_j)}]
        \end{equation}
        and
        \[
            E_G^c f_{\{K,P\}_G} = f_{\{K,P\}_G} E_G^c f_{\{K,P\}_G} = \bigoplus_{i,j =1}^n f_{\{K_i,P_i\}_G} E_G^c f_{\{K_j,P_j\}_G}.
        \] into $k$-submodules. Moreover for $b_{ij} \in k[_{(G,K_i,P_i)}\Gamma _{(G,K_j,P_j)}]$ using the fact we observed in the beginning of the proof its clear that
        \[
            \omega (b_{ij}) = f_{\{K,P\}_G} b_{ij} f_{\{K,P\}_G} = f_{(K_i,P_i)} b_{ij} f_{(K_j,P_j)} \in f_{(K_i,P_i)} E_G^c f_{(K_j,P_j)}.
        \]
        Then each $\omega_{ij}$ is an isomorphism of $k$-modules since $\omega$ is an isomorphism of $k$-modules.
        \item Let $a$ and $b$ be arbitrarily chosen elements in the $k$-algebra $k[\Gamma _{(G,K,P)}]$. Also enumerate $\{K, P\}_G $ as in the previous case, and choose $(K,P) = (K_i,P_i)$. Using the fact that $f_{\{K,P\}_G}$ is central and $f_{(K_i,P_i)}$'s are orthogonal idempotents in $E_G^c$ we have
        \begin{center}
            $\omega _{ii} (a) \omega _{ii} (b) = f_{(K,P)} a f_{(K,P)} b f_{(K,P)} = f_{(K,P)} a f_{\{K,P\}_G} b f_{(K,P)} = f_{(K,P)} f_{\{K,P\}_G} ab f_{(K,P)} = f_{(K,P)} ab f_{(K,P)} = \omega_{ii} (ab).$
        \end{center}
    \end{enumerate}
\end{nothing}

\begin{nothing}\ {\bf Proof of \ref{k-algiso}.} 
In the Proposition \ref{prop:central} we have given a decomposition of $E_G^c$ into two-sided ideals. So it suffices to show that there exists a $k$-algebra isomorphism between $E_G^{c}f_{\{K,P\}_G}$ and Mat$_n (k\Gamma_{(G,K,P)})$. Aiming this we construct a map from  Mat$_n(k\Gamma_{(G,K,P)})$ to $E_G^{c, \{K,P\}}$ and since we already have a $k$-algebra isomorphisms $\omega$, composing this two maps we get desired isomorphism.

  Enumerate the elements of $\{K,P\}_G$ as in the Lemma \ref{lem:w's}. Then chose an element $(T_i,S_i)\in \Sigma _{G\times G}^c$ with $l_0(T_i,S_i) = (K,P)$ and $r_0(T_i,S_i) = (K_i,P_i)$, and set $x_i: = [\frac{G\times G}{S\unlhd T}]$. Denote by $y_i = \frac{1}{|G:P_i|} x_i$ and by $\bar{y_i} = \frac{1}{|G:P|} x_i^{op}$. Then its an easy calculation to show that $y_i\underset{G}{\cdot} \bar{y_i} = e_{(K,P)}$ and $\bar{y_i} \underset{G}{\cdot} y_i = e_{(K_i,P_i)}$. Hence the maps $\sigma _{ij} : k\Gamma_{(G,K,P)} \rightarrow k[_{(G,K_i,P_i)}\Gamma_{(G,K_j,P_j)}], \; a\mapsto \bar{y_i} a y_j$ are $k$-module isomorphisms with the inverse maps $b\mapsto y_i b \bar{y_j}$. Taking the direct sum of the maps $\sigma _{ij}$ and using the direct sum decomposition in (\ref{decomp1}) we obtain a $k$-module isomorphism $\sigma :Mat_n(k\Gamma_{(G,K,P)})\rightarrow E_G^{c, \{K,P\}}$. Infact it is a $k$-algebra isomorphism. So we have a desired isomorphism $\omega \circ \sigma $ of $k$-algebras, which maps diagonal idempotent matrix $e_i$ to $f_{(K_i,P_i)}$.
\end{nothing}

\begin{nothing}{\bf Proof of 6.2}
Let $(T,S)\in \Sigma _{G\times G}$. If $(T,S)$ satisfies $(i)$ then by the decomposition in Proposition \ref{decomp galois} we have $[\frac{G\times G}{S\unlhd T}] \in I_G$. If it satisfies $(ii)$ then $(K,P) := l_0{(T,S)} \not \in \R_G$, which by definition means that $e_{(K,P)} \in I_G$, but then since $I_G$ is an ideal $[\frac{G\times G}{S\unlhd T}] = e_{(K,P)} [\frac{G\times G}{S\unlhd T}] \in I_G$.

Conversely, to prove that every element in $I_G$ can be written as a $k$-linear combination of elements as given in the lemma it suffices to show that $[\frac{G\times H}{S\unlhd T}] \underset{H}{\cdot} [\frac{H\times G}{U\unlhd V}]$ can be written as such a linear combination for any group $|H| < |G|$ and $(T,S)\in \Sigma_{G\times H}$, $(V,U)\in \Sigma_{H\times G}$. By the Mackey formula this product consist of the elements of the form $[\frac{G\times G}{S*U' \unlhd T*V'}]$ for some $(V',U')\in \Sigma _{H\times G}$. Assume that $(T*V',S*U')$ does not satisfy conditions $(i)$ and $(ii)$. So we have that the left invariant of the section $(T*V', S*U')$ equal to $(G,K,P,1)$ for some $(K,P)\in \G$ and $(K,P) \in \R_G$. Proving the similar version of Proposition \cite[2.6]{fibered biset} for the elements of the poset $\G$ one can show that $l(T,S)= (G,K',P',1)$ with $(K',P')\preceq (K,P)$. But then since $\R_G$ is a lower set the pair $(K',P')$ also is reduced and we obtain the contradiction 
    \[
        e_{(K',P')} = \Big[\frac{G\times H}{S\unlhd T}\Big] \underset{H}{\cdot} \Big[\frac{G\times H}{S\unlhd T}\Big]^{op} \in I_G.
    \]
\end{nothing}

\begin{nothing}{\bf Proof of 6.3}
    \begin{enumerate}[(a)]
        \item By the Lemma \ref{lemma8.1} a standart basis element $\big[ \frac{G\times G}{S\unlhd T}\big]$ of $E_G$ belongs to covering algebra if $(T,S)\in \Sigma_{G\times G}^c$ and it belongs to $I_G$ if $(T,S)\not \in \Sigma_{G\times G}^c$ so their sum is a whole endomorphism ring. Thus $E_G^c$ is a covering algebra for $E_G$.
        \item The standart basis elements $\big[\frac{G\times G}{S\unlhd T}\big]$ with $(T,S)\in \Sigma_{G\times G}^c$ such that $l_0(T,S) \not \in \R_G$ generate the intersection of covering algebra and ideal $I_G$ as $k$-module. Hence we have an equality 
        \[
            E_G^c\cap I_G = \bigoplus_{\begin{smallmatrix} \{K,P\}_G\in \G/\sim & \\ (K,P)\not \in \R_G  \end{smallmatrix}} E_G^{c,\{K,P\}}
        \]
        Since the set of reduced elements is a lower set and due to the Lemma \ref{lem:w's} we have the equality
        \[
            \bigoplus_{\begin{smallmatrix} \{K,P\}_G\in \G/\sim & \\ (K,P)\not \in \R_G  \end{smallmatrix}} E_G^{c,\{K,P\}} = \bigoplus_{\begin{smallmatrix} \{K,P\}_G\in \G/\sim & \\ (K,P)\not \in \R_G  \end{smallmatrix}} f_{\{K,P\}_G}E_G^c.
        \]
        Thus we get the desired result.
        \item It is a consequence of the Proposition \ref{prop:central} and Part (b).
        \item It is a consequence of the Proposition \ref{lem:w's}(d)  and Part (c).
        \item Let $[V]$ be the class of simple $k\Gamma_{(G,K,P)}$-module and $(K,P)\in \widetilde{\R}_G$ such that $(K,P) = (K_i,P_i)$ where the elements of $\{K,P\}_G$ are enumerated as in the Lemma \ref{lem:w's}. Then by the Morita equivalence of algebras Mat$_n(k\Gamma_{(G,K,P)})$ and $k\Gamma_{(G,K,P)}$, $[V]$ corresponds to the class of simple module Mat$_n(k\Gamma_{(G,K,P)})e_i\otimes_{k\Gamma_{(G,K,P)}}V$. Then for $y_i = e_i = e_{(K,P)}$ the isomorphism $\omega \circ \sigma$ (from the proof of Theorem \ref{k-algiso}) transports this simple module to the irreducible module $f_{\{K,P\}_G}E_G^cf_{(K,P)}\otimes_{k\Gamma_{(G,K,P)}}V = E_G^cf_{(K,P)}\otimes_{k\Gamma_{(G,K,P)}} V$. The last equality due to the fact that $f_{\{K,P\}_G}$ is commutative in the covering algebra, and $f_{(K,P)}$ are orthogonal idempotents. So by the isomorphism in the Part (c) we are done.
    \end{enumerate}

\end{nothing}
\begin{lem}\label{lem:equalOrder}
Let $G$ and $H$ be finite groups, $(K, P)\in \mathcal R_G$ and $(L, Q)\in \mathcal R_H$ be such that $(G, K, P, 1)\sim (H, L, Q, 1)$. Then $G$ and $H$ have the same order.
\end{lem}
\begin{proof}
Since $(G, K, P, 1)$ and $(H, L, Q, 1)$ are linked, there is a section $[\frac{G\times H}{S\unlhd T}]\in {}_{(G, K, P)}\Gamma_{(H, L, Q)}$. Assume $|G| > |H|$. Then there is a factorization $|H:Q| |G: P|e_{(K, P)} = [\frac{G\times H}{S\unlhd T}]\times_H [\frac{G\times H}{S\unlhd T}]^{op}$. This is impossible since $e_{(K, P)}$ is reduced.
\end{proof}

\begin{lem}\label{lem:technical}
Let $G$ and $H$ be finite groups of the same order, let $\Gamma_k^c(G, H)$ be the submodule of $\Gamma_k(G, H)$ spanned by all sections corresponding to covering pairs. Also let $(K, P)\in \mathcal G_G$ and $(L, Q)\in \mathcal G_H$. Then
\begin{enumerate}[(a)]
\item There is an isomorphism of $(k\Gamma_{(G, K, P)}, k\Gamma_{(H, L, Q)})$-bimodules
\[
k[{}_{(G, K, P)}\Gamma_{(H, L, Q)}] \to f_{(K, P)}\Gamma_k^c(G, H)f_{(L, Q)}
\]
given by mapping $b$ to $f_{(K, P)}b f_{(L, Q)}$. Here the actions on the left and on the right of $f_{(K, P)}\Gamma_k^c(G, H)f_{(L, Q)}$ are given through the isomorphism in Theorem \ref{thm:Gammas}. 
\item Suppose $(L, Q)\in \mathcal R_H$ and $W$ is a simple $k\Gamma_{(H, L, Q)}$-module. Then there is an epimorphism
\[
k[{}_{(G, K, P)}\Gamma_{(H, L, Q)}]\otimes_{k\Gamma_{(H, L, Q)}} W \to f_{(K, P)}S_{H, \widetilde W}(G)
\]
of $k\Gamma_{(G, K, P)}$-modules.
\end{enumerate}
\end{lem}
The proof of this technical lemma is almost identical to the proof of \cite[Lemma 9.5]{fibered biset}. We leave the justification to the reader.

\end{document}